\documentclass[final,12pt]{colt2022} 

\newif\ifconf\conftrue

\ifconf
\title[Minimax Regret for Partial Monitoring: Infinite Outcomes and Rustichini's Regret]
{\normalfont \textsc{Minimax Regret for Partial Monitoring: Infinite Outcomes and Rustichini's Regret}}
\coltauthor{%
 \Name{Tor Lattimore} \Email{lattimore@deepmind.com}\\
 \addr DeepMind, London
}

\usepackage{times}
\else
\title{Minimax Regret for Partial Monitoring: Infinite Outcomes and Rustichini's Regret}
\author{Tor Lattimore}
\fi

\ifconf
\else

\fi

\usepackage{xcolor}
\definecolor{dkblue}{cmyk}{1,.54,.04,.19} 

\usepackage{hyperref} 

\hypersetup{
  bookmarks=true,         
    unicode=false,          
    pdftoolbar=true,        
    pdfmenubar=true,        
    pdffitwindow=false,     
    pdfstartview={FitH},    
    pdftitle={Minimax Regret for Partial Monitoring: Infinite Outcomes and Rustichini's Regret},
    pdfauthor={},     
    pdfsubject={Bandits},   
    pdfcreator={pdflatex},   
    pdfproducer={Producer}, 
    pdfkeywords={bandits} {online learning} {machine learning}, 
    pdfnewwindow=true,      
    colorlinks=true,        
    linkcolor=black,       
    citecolor=dkblue,       
    filecolor=dkblue,       
    urlcolor=dkblue,        
}

\usepackage{nicefrac}
\usepackage{floatrow}

\ifconf
\else
\usepackage{amsmath}
\usepackage{amsthm}
\usepackage{amssymb}
\usepackage{mathtools}
\fi

\usepackage{enumitem}
\usepackage{listings}
\usepackage{tikz}
\usepackage{mathrsfs}
\usepackage{eucal}
\usepackage{times}
\usepackage[boxed]{algorithm}
\usepackage{bm}
\usepackage[capitalise]{cleveref}
\usepackage[bf]{caption}
\usepackage{graphicx}

\newcommand{\ceil}[1]{\left\lceil #1 \right\rceil}
\newcommand{\floor}[1]{\left\lfloor #1 \right\rfloor}

\newenvironment{proofnq}%
{%
 \par\noindent{\bfseries\upshape \proofname\ }%
}%
{\par\bigskip}

\ifconf
\else
\theoremstyle{plain}
\newtheorem{theorem}{Theorem}

\newtheorem{lemma}[theorem]{Lemma}
\newtheorem{proposition}[theorem]{Proposition}

\theoremstyle{definition}
\newtheorem{definition}[theorem]{Definition}

\theoremstyle{remark}

\theoremstyle{definition}
\newtheorem{example}[theorem]{Example}

\fi

\newtheorem{fact}[theorem]{Fact}

\ifconf
\newcommand{\qedhere}{\tag*{$\blacksquare$}}
\fi

\usepackage{natbib}

\newcommand{\R}{\mathbb R}
\newcommand{\N}{\mathbb N}

\newcommand{\Breg}{\mathrm{D}}
\newcommand{\dom}{\operatorname{dom}}
\newcommand{\argmin}{\operatornamewithlimits{arg\,min}}

\newcommand{\ip}[1]{\langle #1 \rangle}
\newcommand{\bip}[1]{\left\langle #1 \right\rangle}
\newcommand{\Reg}{\mathfrak{R}}

\newcommand{\norm}[1]{\Vert #1 \Vert}
\newcommand{\bnorm}[1]{\left\Vert #1 \right\Vert}
\newcommand{\TV}[1]{\norm{#1}_{\text{\normalfont \tiny TV}}}
\newcommand{\E}{\mathbb E}

\newcommand{\lip}{\operatorname{Lip}}
\newcommand{\faces}{\operatorname{faces}}

\newcommand{\cE}{\mathcal E}
\newcommand{\cW}{\mathcal W}
\newcommand{\cK}{\mathcal K}

\newcommand{\cG}{\mathcal G}
\newcommand{\cH}{\mathcal H}
\newcommand{\cP}{\mathcal P}

\newcommand{\cC}{\mathcal C}
\newcommand{\cF}{\mathcal F}
\newcommand{\cA}{\mathcal A}
\newcommand{\cN}{\mathcal N}
\newcommand{\cI}{\mathcal I}

\newcommand{\cZ}{\mathcal Z}
\newcommand{\cS}{\mathcal S}
\newcommand{\cT}{\mathcal T}
\newcommand{\cD}{\mathcal D}
\newcommand{\cL}{\mathcal L}
\newcommand{\sP}{\mathscr P}
\newcommand{\sL}{\mathscr L}

\newcommand{\I}{\mathrm{I}}

\newcommand{\V}{\mathcal V}
\newcommand{\Vm}{\mathcal V^m}

\newcommand{\zeros}{ \bm 0}
\newcommand{\bbP}{\mathbb P}

\newcommand{\const}{\operatorname{const}}
\newcommand{\ones}{\bm{1}}

\newcommand{\conv}{\operatorname{conv}}
\newcommand{\ext}{\operatorname{ext}}

\newcommand{\sind}{\bm{1}}

\newcommand{\ri}{\operatorname{ri}}
\newcommand{\dist}{\operatorname{dist_1}}

\newcommand{\diam}{\operatorname{diam}}

\newcommand{\spt}{\operatorname{spt}}

\newcommand{\KL}{\operatorname{KL}}
\renewcommand{\d}[1]{\operatorname{d}\!#1}

\newcommand{\val}{\textrm{V}}
\newcommand{\rel}{\textrm{R}}

\begin{document}

\maketitle


\begin{abstract}
We show that a version of the generalised information ratio of \cite{LG20} determines the asymptotic minimax regret for all finite-action partial monitoring games
provided that \textit{(a)} the standard definition of regret is used but the latent space where the adversary plays is potentially infinite; or \textit{(b)} the regret introduced by 
\cite{Rus99} is used and the latent space is finite. 
Our results are complemented by a number of examples. For any $p \in [\nicefrac{1}{2},1]$ there exists an infinite partial monitoring game for which the minimax regret over $n$ rounds is $n^p$ up to subpolynomial factors
and there exist finite games for which the minimax Rustichini regret is $n^{4/7}$ up to subpolynomial factors.
\end{abstract}

\ifconf
\fi

\section{Introduction}

The minimax asymptotic regret for partial monitoring is only understood for finite games and where the notion of regret compares the learner's accumulated losses
to the loss of the best action in hindsight. For many partial monitoring games the learner is at a severe disadvantage compared to the adversary because the signal function
may be relatively uninformative. The notion of regret introduced by \cite{Rus99} serves as a better measure of performance for these games, as it measures the learner's performance
compared to a baseline that is achievable, while the standard notion of regret may be linear.
Our contribution is twofold. We show that the generalised information ratio of \cite{LG20} determines the asymptotic
minimax regret for all finite-action games where either \textit{(a)} the standard regret is used but the latent space where the adversary plays
is possibly infinite; or \textit{(b)} the definition of the regret
proposed by \cite{Rus99} is used and the latent space is finite.
To state our main theorem informally, for $\lambda > 1$ let
\begin{align*}
\Psi^\lambda_\star = \sup_{\textrm{prior}} \min_{\textrm{policy}} \frac{[\textrm{expected regret}]^\lambda}{\textrm{information gain}}
\end{align*}
be the minimax information ratio, where the supremum is over priors on the adversary's decision for one round and the infimum is over distributions on actions.
The information gain and regret will of course be specified later, with the latter bounded in $[0,1]$, which means that $\lambda \mapsto \Psi^\lambda_\star$ is non-increasing. 
Let $\lambda_\star = \inf\{\lambda > 1 : \Psi^\lambda_\star < \infty\}$.
We prove that
\begin{align*}
\limsup_{n\to\infty} \frac{\log(\Reg_n^\star)}{\log(n)} = 1 - \frac{1}{\lambda_\star}\,,
\end{align*}
where $\Reg_n^\star$ is the minimax regret over $n$ rounds.
Because of the logarithm and the limit, our result says nothing about the leading constant, or even subpolynomial factors.
In exchange, the result holds in extraordinary generality. Only finiteness of the action set is needed.
An interesting consequence of our proofs is that the lower bound is always attained by a stationary stochastic adversary.
As illustrated by our construction in \cref{sec:exotic}, this is a powerful result because it allows you to prove upper bounds for the adversarial setting by
constructing algorithms for the stochastic setting.

\paragraph{Setting}
Let $[k] = \{1,\ldots,k\}$.
A $k$-action partial monitoring game is determined by a latent space $\cZ$, a loss function $\cL : [k] \times \cZ \to [0,1]$ and a signal function
$\cS : [k] \times \cZ \to \Sigma$ where $\Sigma$ is called the signal set.
Let $\cK = \sP(\cZ)$ be the space of finitely supported distributions on $\cZ$ 
and extend the second argument of $\cL$ linearly to $\cK$ by $\cL(a, x) = \int_{\cZ} \cL(a, z) \d{x}(z)$ for $x \in \cK$.
The condition that distributions in $\cK$ have finite support is a useful trick to avoid measure-theoretic challenges.
The learner and adversary interact over $n$ rounds.
At the beginning the adversary secretly chooses a sequence of distributions $(x_t)_{t=1}^n$ in $\cK$ and then $(z_t)_{t=1}^n$ are sampled from the product measure 
$\otimes_{t=1}^n x_t$ but not observed.
The learner sequentially takes actions $(a_t)_{t=1}^n$ in $[k]$ and observes signals $(\sigma_t)_{t=1}^n$ with $\sigma_t = \cS(a_t, z_t)$.
To be clear, $a_t$ is chosen based on $a_1,\sigma_1,\ldots,a_{t-1},\sigma_{t-1}$, possibly with randomisation.
The loss suffered by the learner in round $t$ is $\cL(a_t, z_t)$, which is not observed. 
The regret of a learner $\sL$ over $n$ rounds is
\begin{align*}
\Reg_n(\sL, (x_t)_{t=1}^n) = \E\left[\sum_{t=1}^n \V(a_t, x_t)\right] - \V_\star\left(\frac{1}{n} \sum_{t=1}^n x_t\right)\,,
\end{align*}
where $\V : [k] \times \cK \to \R$ and $\V_\star : \cK \to \R$ are functions for which we will consider two possible definitions
and the expectation is with respect to the randomness in the learner's actions and the outcomes $(z_t)$.
The standard choice of $\V$ and $\V_\star$ is
\begin{align}
\V(a, x) &= \cL(a, x)  &
\V_\star(x) &= \min_{a \in [k]} \cL(a, x)\,.
\label{eq:standard}
\end{align}
With these definitions the regret compares the cumulative expected loss of the learner to the cumulative loss of the best action in hindsight.
As we explain in a moment, the regret definition arising from the definitions in \cref{eq:standard} can be too demanding for games where the signal function
provides limited information to the learner. Before that, we need a little more notation.
The minimax regret is
\begin{align*}
\Reg_n^\star = \inf_{\sL} \sup_{(x_t)_{t=1}^n} \Reg_n(\sL, (x_t)_{t=1}^n)\,,
\end{align*}
which of course depends on the structure of the game $(\cL, \cS)$ and the definition of $\V$.
If $\cZ$ is finite and $\V$ is defined by \cref{eq:standard}, then the classification of finite partial monitoring \citep{BFPRS14} states that the minimax regret falls into one of four categories: 
$\Reg_n^\star \in \{0, \Theta(n^{1/2}), \Theta(n^{2/3}), \Theta(n)\}$. 
Partial monitoring games where $\Reg_n^\star = \Theta(n)$ are called hopeless. These are games where the signal function does not provide enough information for the learner
to identify an optimal action and it is these games where an alternative definition of $\V$ is more natural.

\begin{example}[\textsc{matching pennies in the dark}, \citealt{Per11}]\label{example:pennies}
Let $k = 2$ and $\cZ = [2]$ and $\cL(a, z) = \sind_{a = z}$ and
$\cS(a, z) = \bot$ for all $a \in [k]$ and $z \in \cZ$. Learning is hopeless in this game because the learner's observation does not depend on $z$. There is no way to 
differentiate between an adversary that plays a Dirac on $z_t = 1$ in every round and one that plays $z_t = 2$ in every round. Since the optimal
action is different in each case, the learner suffers linear minimax regret.
\end{example}

In games like matching pennies in the dark, the expectation that the learner compete with the optimal action in hindsight is too strict.
\cite{Rus99} proposed to compare the learner's performance to the Nash equilibrium of the game where the learner knows the average distribution of the signal for each
action, which can be estimated by playing all actions. 
When $x \in \cK$ and $a \in [k]$, let $\cS(a, x)$ be the law of $\cS(a, z)$ when $z \sim x$ and
define an equivalence relation $\rel$ on $\cK \times \cK$ by $x \rel y$ if $\cS(a, x) = \cS(a, y)$ for all $a \in [k]$.
That is, $x \rel y$ is true if the signal distribution for all actions is the same for both $x$ and $y$.
The Rustichini notion of regret uses
\begin{align}
\V(\pi, x) &= \sup_{y : y \rel x} \cL(\pi, y)
&
\V_\star(x) &= \min_{\pi \in \Delta_k} \V(\pi, x)\,.
\label{eq:rustichini}
\end{align}
where $\Delta_k = \{\pi \in [0,1]^k : \norm{\pi}_1 = 1\}$ and $\cL(\pi, x) = \sum_{a \in [k]} \pi(a) \cL(a, x)$.
Note that the minimising policy in \cref{eq:rustichini} may not be an extreme point of the simplex.
Indeed, in the game of \cref{example:pennies}, $x \rel y$ for all $x$ and $y$. An easy calculation shows that $\V_\star(x) = 1/2$ and the policy $\pi$ such that $\V(\pi, x) = \V_\star(x)$ randomises uniformly between the two actions.
The standard and Rustichini regret coincide when $x \rel y$ if and only if $x = y$.
Otherwise the situation is quite subtle, as we explain at length in \cref{sec:compare}, a section that is best read after the main body.

\paragraph{Important remark} When we say a result holds in the \textit{standard setting} we mean that $\V$ and $\V_\star$ are defined as in \cref{eq:standard} 
and $\cZ$ may be finite or infinite, while in the \textit{Rustichini setting} $\V$ and $\V_\star$ are
defined as in \cref{eq:rustichini} and $\cZ$ is finite.
Occasionally we use the notation
$\const_{\cL\cS}$ to denote a constant that depends only on $\cL$ and $\cS$ and may vary from one equation to the next.  
Without the indices, $\const$ refers to a universal constant.

\paragraph{Related work}
Partial monitoring and the non-standard definition of the regret used in part of this work was proposed by \cite{Rus99}, who proved a Hannan-type consistency result
using an argument based on Blackwell approachability \citep{Bla56}.
More recent work on partial monitoring has focussed on finite games using the classical definition of the regret, where the learner competes with the best action in hindsight.
A long line of work covering various cases \citep{PS01,CBLuSt06,FR12,ABPS13,LS18pm} led to a complete classification of finite partial monitoring
games \citep{BFPRS14}. This shows that the minimax regret for all such games is either $0$, $\Theta(n^{1/2})$, $\Theta(n^{2/3})$ or $\Theta(n)$, with the leading constants
and rate depending on the structure of the game.
More recent work has developed simple algorithms and improved dependence on constants \citep{LS19pminfo,LS19pmsimple}.
As far as we know, there has been no general study of infinite partial monitoring games, though of course many specific instances fall into this category (bandits
with a continuum of possible losses, for example).

Besides \cite{Rus99}, very few authors have directly considered partial monitoring with Rustichini's regret.
\cite{LMS08} improve on the results of \cite{Rus99} by providing rates and high probability bounds, which in the general case are $O(n^{4/5})$ with certain special
cases giving improved bounds. 
The setting can be viewed as an instance of the even more general problem of Blackwell approachability with partial monitoring studied by \cite{Per11b,MS11,MPS14} 
and \cite{KP17}. When applied to our setting, the last of these papers prove that the minimax regret is $\Reg_n^\star = O(n^{2/3})$. There are games for 
which $\Reg_n^\star = \Omega(n^{2/3})$, so this rate cannot be improved in general.
Note this means there is no concept of a hopeless game when the Rustichini regret is used (essentially this is why it is used).
An interesting question is whether or not our ideas can be lifted to approachability with partial monitoring.

Our approach for proving upper bounds is based on the duality between information directed sampling \citep{RV14} and its connection to quantities involved
in bounding the regret of mirror descent \citep{ZL19,LG20}.
The principle novelty here is new lower bounds in terms of the information ratio.
While one might guess that the relative entropy might be the notion of information gain most suitable for this, it turns out that defining the information
gain as an expected Bregman divergence between a posterior and prior with respect to the logarithmic barrier (rather than negentropy) is easier
to work with. Correspondingly, our algorithm is based on mirror descent with the logarithmic barrier as a potential function.

A recent paper by \cite{FKQR21} investigates a related idea and also prove upper and lower bounds in terms of an information-ratio-like quantity.
Their results are stronger in some dimensions and weaker in others. For example, while we consider adversarial partial monitoring problems, they focus on stochastic problems
with bandit feedback (but can handle contextual problems). On the other hand, our work essentially relies on a (small) finite action set 
and only characterises the minimax regret in terms of the horizon, while
their methods are suitable for large action sets and yield bounds that demonstrate dependencies on other game-dependent quantities.
Finally, our lower bound controls the expected regret, while their lower bound is a kind of medium probability result and does not yield meaningful lower bounds on the expected regret.
One might hope to one day get the best of all these worlds, though at the moment it seems some real new idea is needed for such a result.

\paragraph{Notation}
The vectors of all ones and zeros are $\ones$ and $\zeros$ respectively. The dimension of these quantities should understood from the context.
Let $\sP(A)$ denote the set of finitely supported probability distributions on a set $A$ and $\sP_m(A) \subset \sP(A)$ the set of those
distributions supported on at most $m$ elements. 
The support of a distribution $\mu \in \sP(A)$ is $\spt(\mu) = \{z \in A : \mu(z) > 0\}$.
The relative interior of a set $A \subset \R^d$ is $\ri(A)$ and its convex hull is $\conv(A)$.
When $A$ is convex, its exterior points are denoted by $\ext(A)$.
The indicator function of $A$ is $\sind_A$.
When $A$ is a polytope, $\faces(A)$ is the set of faces of $A$, including itself and the empty set (formal definition in \cref{sec:poly}).
The relative entropy between probability measures $\nu$ and $\mu$ on the same measurable space is $\KL(\nu, \mu)$ and the total variation distance is $\TV{\nu - \mu}$.
The signal function is extended linearly to policies by $\cS(\pi, x) = \sum_{a \in [k]} \pi(a) \cS(a, x)$ and
its second argument is extended further to measures $\mu \in \sP(\cK)$ by $\cS(\pi, \mu) = \cS(\pi, \int x \d{\mu}(x))$.
The $\ell_1$ distance between a point $x \in \R^d$ and a nonempty closed set $F \subset \R^d$ is $\dist(x, F) = \min_{y \in F} \norm{x - y}_1$.
A game is called finite if $\cZ$ and $k$ are finite. For finite games the loss and signal functions are often represented by matrices
where the adversary plays in the columns and the learner plays in the rows: $\cL_{a,z} = \cL(a, z)$ and $\cS_{a,z} = \cS(a, z)$. 
Given a convex function $F : \R^m \to \R \cup \{\infty\}$, the Bregman divergence with respect to $F$ is $\Breg(x, y) = F(x) - F(y) - \ip{\nabla F(y), x - y}$, which is defined 
provided that $x \in \dom(F)$ and $y \in \dom(\nabla F)$.
The Fenchel conjugate of $F$ is $F^\star(x) = \sup_{u \in \R^m} \ip{u, x} - F(x)$.

\paragraph{Properties of $\bm{\V}$}
The functions $\V$ and $\V_\star$ play a central role in our analysis.

\begin{lemma}\label{lem:V}
The following hold in both the standard and Rustichini settings:
\begin{enumerate}[noitemsep,nolistsep]
\item [(a)] $\V_\star$ is concave and piecewise linear.
\item [(b)] $x \mapsto \V(\pi, x)$ is concave and piecewise linear for all $\pi \in \Delta_k$ (linear in the standard setting).
\item [(c)] $\pi \mapsto \V(\pi, x)$ is convex and piecewise linear for all $x \in \cK$ (linear in the standard setting).
\end{enumerate}
\end{lemma}

The proof is elementary in the standard setting. In the Rustichini setting it follows from classical linear programming arguments.
Details are in Appendix~\ref{app:tech}.
This lemma is one place where the finiteness of $\cZ$ is used in our results on the Rustichini notion of regret.
Let $m$ be the smallest number of pieces needed to define $\V_\star$ and let $\Vm : \cK \to \R^m$ be linear so that $\V_\star(x) = \min_{p \in \Delta_m} \ip{p, \Vm(x)}$.
In the standard setting $m \leq k$ is immediate while in the Rustichini setting
$m$ can be larger depending on the structure of the signal function but is finite whenever $\cZ$ is finite.
Note that $\max_{x \in \cK} \norm{\Vm(x)}_\infty < \infty$ is obvious for both settings.
The constant $m$ is used to denote the number of linear pieces needed to represent $\V_\star$ for the remainder of the article.

\section{Mirror descent, duality and the information ratio}

We are now in a position to formally define the information ratio and state our main theorem.
Given a prior $\mu \in \sP(\cK)$, let
\begin{align*}
\I(\pi, \mu) &= \sum_{a=1}^k \pi(a) \sum_{x \in \spt(\mu)} \KL\left(\cS(a,\mu), \cS(a,x)\right) \,,
\end{align*}
which will be used as a measure of information gain.
At a very high level, $\I(\pi, \mu)$ is a measure of the variation of signal distributions over different $x \in \spt(\mu)$ when an action $a$ is sampled from $\pi$.
This definition differs from the version of information gain based on the mutual information introduced by \cite{RV14} and used in most followup works.
Instead, it corresponds to using
the logarithmic barrier in the generalised machinery by \cite{LG20}. Very briefly, the mutual information is defined as an expected Bregman divergence between posterior and
prior with respect to the negentropy potential. 
Our information gain can be written as an expected Bregman divergence 
with respect to the logarithmic barrier (\cref{sec:thm:exp-by-opt}).
We also need to define an instantaneous regret estimate:
\begin{align}
\Delta(\pi, \mu) &= \int_{\cK} (\V(\pi, x) - \V_\star(x)) \d{\mu}(x) \,. 
\label{eq:regret}
\end{align}
We will constantly abuse notation by writing $\Delta(\pi, x)$ as an abbreviation for $\Delta(\pi, \mu)$ with $\mu$ a Dirac on $x$.
Finally, let $\Delta_\star(\mu) = \min_{\pi \in \Delta_k} \Delta(\pi, \mu)$, which exists by continuity and compactness arguments (Lemma~\ref{lem:exist}).

\begin{fact}
$\pi \mapsto \Delta(\pi, \mu)$ is convex, $\mu \mapsto \Delta(\pi, \mu)$ is linear and $\pi \mapsto \I(\pi, \mu)$ is linear.
\end{fact}

The $\lambda$-information ratio is a function $\Psi^\lambda : \Delta_k \times \sP(\cK)$ is given by
\begin{align*}
\Psi^\lambda(\pi, \mu) = \frac{\Delta(\pi, \mu)^\lambda}{\I(\pi, \mu)}\,,
\end{align*}
where $0/0 = 0$.
Let $\Psi^\lambda_\star(\mu) = \min_{\pi \in \Delta_k} \Psi^\lambda(\pi, \mu)$. The minimax information ratio is
$\Psi_\star^\lambda = \sup_{\mu \in \sP_m(\cK)} \Psi_\star^\lambda(\mu)$. 
Note the restriction to measures $\mu \in \sP_m(\cK)$ with support size at most $m$, which is the number of pieces needed to represent piecewise linear function $\V_\star$.
Since $\Delta(\pi, \mu) \in [0,1]$, the map $\lambda \mapsto \Psi_\star^\lambda$ is nonincreasing.
Let $\lambda_\star = \inf\{\lambda > 1 : \Psi_\star^\lambda < \infty\}$, where the infimum of the empty set is $\infty$.

\begin{theorem}[\textsc{main theorem}]\label{thm:main}
In both the standard and Rustichini settings: 
\begin{align*}
\limsup_{n\to\infty} \frac{\log(\Reg_n^\star)}{\log(n)} = 1 - \frac{1}{\lambda_\star} \,.
\end{align*}
\end{theorem}

Note that \cref{thm:main} does not say that the regret does not depend on whether or not the standard or Rustichini regret is used because $\V$ appears in 
the definition of the regret function $\Delta$ and hence also the information ratio.
We prove \cref{thm:main} by establishing upper and lower bounds in terms of the information ratio, starting with the upper bound.

\paragraph{Algorithm and upper bound}
Our algorithm is a modification of the exploration-by-optimisation algorithm \citep{LG20}, which itself is a version of mirror descent
with loss estimators and exploration distribution computed by solving an optimisation problem.
The mirror descent component plays on the domain $\cD = \Delta_m \cap [1/n, 1]^m$ and using as a potential function the 
logarithmic barrier $F : \R^m \to \R$ defined on the positive orthant by $F(q) = -\sum_{i=1}^m \log(q_i)$.
Let $F^\star$ be the Fenchel conjugate of $F$ and let $\Breg$ and $\Breg_\star$ be the Bregman divergences with respect to $F$ and $F^\star$.
Let $\cE$ be the space of all functions from $[k] \times \Sigma \to \R^m$.
Given $q \in \cD$ and $\eta > 0$, define $\Lambda_{q\eta} : \ri(\Delta_k) \times \cE \to \R$ by
$\Lambda_{q \eta}(\pi, \hat v)=$
\begin{align*}
\sup_{\substack{p \in \cD \\ x \in \cK}} \Bigg(\V(\pi, x) - \ip{p, \Vm(x)} 
 + \E\left[\frac{\ip{p - q, \hat v(a, \sigma)}}{\pi(a)}
+ \frac{1}{\eta} \Breg_\star\left(\nabla F(q) - \frac{\eta \hat v(a, \sigma)}{\pi(a)}, \nabla F(q)\right)\right]\Bigg)\,,
\end{align*}
where the expectation is with respect to $(z, a) \sim x \otimes \pi$ and the signal is $\sigma = \cS(a, z)$.
Then let $\Lambda_{q \eta}^\star = \inf_{\pi \in \ri(\Delta_k), \hat v \in \cE} \Lambda_{q \eta}(\pi, \hat v)$.
At a high level, $\Lambda_{q\eta}$ is a quantity that appears in the regret bound of mirror descent.
Our algorithm will minimise this bound.

\begin{algorithm}
\begin{lstlisting}[mathescape=true]
input: learning rate $\eta > 0$ and precision $\epsilon > 0$ and horizon $n$
$q_1 = \ones / m \in \Delta_m$
for $t = 1$ to $n$:
  find $\pi_t \in \ri(\Delta_k)$ and $\hat v_t \in \cE$ such that $\Lambda_{q_t\eta}(\pi_t, \hat v_t) \leq \Lambda_{q_t\eta}^\star + \epsilon$
  sample $a_t \sim \pi_t$ and observe $\sigma_t$
  update: $\displaystyle q_{t+1} = \argmin_{q \in \cD} \frac{\ip{q, \hat v(a_t, \sigma_t)}}{\pi_t(a_t)} + \frac{1}{\eta} \Breg(q, q_t)$
\end{lstlisting}
\caption{Exploration by optimisation}\label{alg:exp-by-opt}
\end{algorithm}

In the standard setting \cref{alg:exp-by-opt} is the same as the original.
In the Rustichini setting the
difference is that our algorithm plays mirror descent on $\Delta_m$ rather than $\Delta_k$ and $\pi\mapsto \V(\pi, x)$ is convex, rather than linear.
The upper bound in \cref{thm:main} follows almost immediately from \cref{thm:exp-by-opt} below, the proof of which is
essentially a corollary of Theorems 8 and 9 in the paper by \cite{LG20}. The differences and additional steps are explained in \cref{sec:thm:exp-by-opt}.

\begin{theorem}\label{thm:exp-by-opt}
Provided that $\eta > 0$, $n \geq m$ and $\epsilon \geq 0$, then for any $\lambda > 1$ the regret of \cref{alg:exp-by-opt} is bounded by 
\begin{align*}
\Reg_n \leq n\epsilon + 2 \max_{x \in \cK} \norm{\Vm(x)}_\infty  
+ \frac{m \log(n/m)}{\eta} 
+ n\left(\frac{\lambda - 1}{\lambda}\right)\left(\frac{\eta \Psi_\star^\lambda}{\lambda}\right)^{\frac{1}{\lambda - 1}} \,.
\end{align*}
\end{theorem}

The upper bound of Theorem~\ref{thm:main} is a corollary of Theorem~\ref{thm:exp-by-opt}.

\begin{proof}[\textsc{theorem~\ref{thm:main}, upper bound}]
If $\lambda_\star = \infty$, then the result is obvious. Suppose that $\lambda_\star < \lambda < \infty$. Then $\Psi_\star^\lambda < \infty$ and
by \cref{thm:exp-by-opt}, if \cref{alg:exp-by-opt} is run with learning rate
\begin{align}
\eta = \lambda \left(\frac{m \log(n/m)}{n}\right)^{1 - 1/\lambda} (\Psi_\star^\lambda)^{-1/\lambda}
\label{eq:eta}
\end{align}
and $\epsilon = 1/n$,
then 
\begin{align*}
\Reg_n 
&\leq 1 + 2 \max_{x \in \cK} \norm{\Vm(x)}_\infty + \frac{m \log(n/m)}{\eta} + n\left(\frac{\lambda - 1}{\lambda}\right)\left(\frac{\eta \Psi_\star^\lambda}{\lambda}\right)^{\frac{1}{\lambda - 1}} \\
&= 1 + 2 \max_{x \in \cK} \norm{\Vm(x)}_\infty + \left(m \Psi_\star^\lambda \log(n/m) \right)^{1/\lambda} n^{1 - \frac{1}{\lambda}}\,.
\end{align*}
Since this bound is independent of the adversary's decisions, $\limsup_{n\to\infty} \log(\Reg_n^\star) / \log(n) \leq 1 - \frac{1}{\lambda}$.
The theorem follows by taking the limit as $\lambda$ tends to $\lambda_\star$ from above.
\end{proof}

\paragraph{Computation complexity}
The computation complexity of \cref{alg:exp-by-opt} depends on the structure of the game. 
For finite $\cZ$, there are two challenges. First, some pre-computation is needed to find $\Vm$, which is possible using linear
programming.
The optimisation to find $\pi_t$ and $\hat v_t$ is a convex program and can also be solved to reasonable precision in polynomial time.
Updating $q_t$ is naively a convex program as well, but can be solved analytically except for a line search.
The optimal tuning of $\eta$ given in \cref{eq:eta} is probably hard to find but there is no problem tuning the learning rate online as explained by \cite{LG20}.  
All up, for medium-sized Rustichini games \cref{alg:exp-by-opt} is a reasonably practical option. For infinite games the optimisation problem to find $\pi_t$ and $\hat v_t$ remains
convex but is infinite-dimensional. In this case we view the arguments as more of a tool
for understanding the regret and an inspiration for algorithm design.

\section{Lower bounds}

More or less all lower bounds for bandits and online learning are established via information-theoretic techniques. 
Our lower bound (below) is proven similarly but because of the limited structure of the problem only the most general machinery can be used.
The first lemma is used to relate the instantaneous regret to the minimax information ratio in a linear way. A similar result when $\lambda = 2$ was given by \cite{KLVS21}. 
The proof is in Appendix~\ref{app:tech}.

\begin{lemma}\label{lem:psi-mix}
For all $\mu \in \sP(\cK)$ and $\pi \in \Delta_k$, $\Psi^\lambda_\star(\mu) \leq \frac{2^\lambda \Delta_\star(\mu)^{\lambda - 1} \Delta(\pi, \mu)}{\I(\pi, \mu)}$.
\end{lemma}

Next we need a result showing that for any $\mu \in \sP_m(\cK)$, if a collection of policies on index set $\spt(\mu)$ are sufficiently close to each other, then  
at least one of the policies has regret comparable to $\Delta_\star(\mu)$. 

\begin{lemma}\label{lem:key}
There exists a game-dependent constant $\varphi > 0$ such that the following holds.
Let $\mu \in \sP_m(\cK)$ and suppose that $(\pi_x)_{x \in \spt(\mu)}$ is a set of policies with $\norm{\pi_x - \pi}_\infty \leq \varphi$ for some $\pi \in \Delta_k$.
Then $\max_{x \in \spt(\mu)} \Delta(\pi_x, x) \geq \const_{\cL\cS} \Delta_\star(\mu)$.
\end{lemma}

\begin{proofnq}[\textsc{standard setting}]
We prove Lemma~\ref{lem:key} in the standard setting with $\varphi = 1/(4k)$.
Let $\mu \in \sP_m(\cK)$ and $\pi$ and $(\pi_x)_{x \in \spt(\mu)}$ be a collection of policies satisfying the conditions of the lemma.
Let $\cA = \{a \in [k] : \pi(a) \geq 1/(2k)\}$, which satisfies $\sum_{a \in \cA} \pi(a) \geq 1/2$.
Define a policy $\rho(a) = \pi(a) \sind_\cA(a) / \sum_{a \in \cA} \pi(a)$.
Then,
\begin{align*}
\Delta_\star(\mu) 
&\leq \Delta(\rho, \mu) 
= \frac{\int_\cK \sum_{a \in \cA} \pi(a) \Delta(a, x) \d{\mu}(x)}{\sum_{a \in \cA} \pi(a)}  
\leq 2 \int_\cK \sum_{a \in \cA} \pi(a) \Delta(a, x) \d{\mu}(x) \\
&\leq 2 \max_{x \in \spt(\mu)} \sum_{a \in \cA} \pi(a) \Delta(a, x)  
\leq 4 \max_{x \in \spt(\mu)} \sum_{a \in \cA} \pi_x(a) \Delta(a, x) 
\leq 4 \max_{x \in \spt(\mu)} \Delta(\pi_x, x)\,. 
\qedhere
\end{align*}
\end{proofnq}

The proof above relied on the linearity of $\pi \mapsto \Delta(\pi, x)$, which does not hold in the Rustichini setting.
There may be an elegant linear programming proof of Lemma~\ref{lem:key} for the Rustichini setting, but we offer a relatively involved geometric argument in \cref{sec:lem:key}.
Now we have all the tools needed to prove the lower bound in our main theorem.

\begin{proof}[\textsc{theorem~\ref{thm:main}, lower bound}]
Our proof works by showing that for every learner there exists a stochastic adversary
that samples $(z_t)_{t=1}^n$ independently from some carefully chosen $x \in \cK$ for which the regret is large.
We start by introducing some notation (step 1). In step 2 we make the necessary relative entropy calculations and connect the information gain to the relative entropy
between various stochastic environments.
The regret is then decomposed in two ways, depending on whether or not the algorithm explores a lot or a little (step 3). Finally we balance the constants in the
proof to obtain the result.

\paragraph{Step 1: Notation and setup}
The result is obvious if $\lambda_\star = 1$. For the remainder assume that $\lambda_\star > 1$.
Fix an arbitrary policy for the learner and
let $1 < \lambda < \lambda_\star$ and $\mu \in \sP_m(\cK)$.
Given $x \in \cK$, let $\bbP_x$ be the measure on $(a_t)_{t=1}^n$, $(\pi_t)_{t=1}^n$ and $(\sigma_t)_{t=1}^n$ when the learner interacts with the (stochastic) environment where $x_t = x$ for all $t$. 
Let $\bbP = \bbP_y$ with $y = \int_\cK x \d{\mu}(x)$. 
Expectations with respect to these measures are  denoted by $\E_{\bbP}$ and $\E_{\bbP_x}$.
The regret when the learner faces the stochastic environment with $(z_t)_{t=1}^n$ sampled from $\otimes_{t=1}^n x$ is denoted by
\begin{align*}
\Reg_n(x) = \E_{\bbP_x}\left[\sum_{t=1}^n \V(\pi_t, x)\right] - n \V_\star(x)\,.
\end{align*}
Let $\cF_t = \sigma(\pi_1,a_1,\sigma_1,\ldots,\pi_t,a_t,\sigma_t)$ be the $\sigma$-algebra generated by the history after $t$ observations.
Given an $(\cF_t)_{t=0}^n$-adapted stopping time $\tau$, let $\bbP^\tau$ and $\bbP_x^\tau$ denote the measures $\bbP$ and $\bbP_x$ 
restricted to $\cF_\tau$.

\paragraph{Step 2: Relative entropy}
Let $\tau$ be a stopping time with respect to the filtration $(\cF_t)_{t=0}^n$ with $\tau \in \{0,\ldots,n\}$ almost surely. 
Then by the chain rule for relative entropy,
\begin{align}
\KL(\bbP^\tau, \bbP_x^\tau) 
&= \E_{\bbP}\left[\sum_{t=1}^\tau \KL(\cS(a_t, \mu), \cS(a_t, x))\right] 
\leq \E_{\bbP}\left[\sum_{t=1}^\tau \sum_{y \in \spt(\mu)} \KL(\cS(a_t, \mu), \cS(a_t, y))\right] \nonumber \\
&= \E_{\bbP}\left[\sum_{t=1}^\tau \I(\pi_t, \mu)\right]\,. \label{eq:kl-rel1}
\end{align}
We also need a Pinsker-like inequality, which we prove in Appendix~\ref{app:tech}.

\begin{lemma}\label{lem:pinsker}
Suppose that $x \in \spt(\mu)$ and $\epsilon > 0$ and $t \in \{0,\ldots,n\}$. Then,
\begin{align*}
\TV{\bbP_x^t - \bbP^t} \leq \sqrt{2\epsilon} + \bbP\left(\sum_{s=1}^t \I(\pi_s, \mu) \geq \epsilon\right)\,.
\end{align*}
\end{lemma}

\paragraph{Step 3: Regret decomposition}
Let $\varphi$ be the constant from Lemma~\ref{lem:key} and
$\delta = \varphi/(1+1/m)$ and $\epsilon = (\delta/m)^2 / 2$.
Suppose that
\begin{align}
\bbP\left(\sum_{t=1}^n \I(\pi_t, \mu) \geq \epsilon\right) \leq \delta\,.
\label{eq:lower-1}
\end{align}
By Lemma~\ref{lem:pinsker}, for all $t \in [n]$,
$\bnorm{\E_{\bbP_x}[\pi_t] - \E_{\bbP}[\pi_t]}_\infty \leq \sqrt{2\epsilon} + \delta = \varphi$.
Hence, by a triangle inequality,
\begin{align*}
\bnorm{\frac{1}{n} \sum_{t=1}^n \E_{\bbP_x}[\pi_t] - \frac{1}{n} \sum_{t=1}^n \E_{\bbP}[\pi_t]}_\infty 
\leq \frac{1}{n} \sum_{t=1}^n \bnorm{\E_{\bbP_x}[\pi_t] - \E_{\bbP}[\pi_t]}_\infty \leq \varphi\,.
\end{align*}
Therefore, by Lemma~\ref{lem:key} there exists an $x \in \spt(\mu)$ such that
\begin{align*}
\Delta\left(\frac{1}{n} \sum_{t=1}^n \E_{\bbP_x}[\pi_t], x\right) \geq \const_{\cL\cS} \Delta_\star(\mu) \,.
\end{align*}
Combining this with the convexity of $\pi \mapsto \Delta(\pi, x)$,
\begin{align*}
\Reg_n^\star
&\geq \max_{x \in \spt(\mu)} \Reg_n(x) 
= \max_{x \in \spt(\mu)} \sum_{t=1}^n \E_{\bbP_x}\left[\Delta(\pi_t, x)\right] \\
&\geq n \max_{x \in \spt(\mu)} \Delta\left(\frac{1}{n} \sum_{t=1}^n \E_{\bbP_x}[\pi_t], x\right) 
\geq \const_{\cL\cS} n \Delta_\star(\mu)\,.
\end{align*}
Suppose for the remainder of this step that \cref{eq:lower-1} does not hold.
Let $\tau = \max\{t \in [n] : \sum_{s=1}^t \I(\pi_s, \mu) \leq \epsilon\}$, which is a stopping time because $\I(\pi_{t+1}, \mu)$ is $\cF_t$-measurable.
By Lemma~\ref{lem:psi-mix},
\begin{align*}
\bbP\left(\tau < n \text{ and } \sum_{t=1}^{\tau+1} \Delta(\pi_t, \mu) \geq \frac{2^{-\lambda} \epsilon \Psi^\lambda_\star(\mu)}{\Delta_\star(\mu)^{\lambda - 1}}\right) 
\geq \bbP\left(\tau < n\right)
> \delta\,.
\end{align*}
Hence, by a union bound, there exists an $x$ such that
\begin{align*}
\bbP\left(\tau < n \text{ and } \sum_{t=1}^{\tau+1} \Delta(\pi_t, x) \geq \frac{2^{-\lambda} \epsilon \Psi^\lambda_\star(\mu)}{\Delta_\star(\mu)^{\lambda - 1}}\right) > \frac{\delta}{m}\,.
\end{align*}
Since $\KL(\bbP^{\tau}, \bbP_x^{\tau}) \leq \epsilon$ by the definition of $\tau$ and \cref{eq:kl-rel1} and $\Delta(\pi_t, x)$ is $\cF_{t-1}$-measurable, by Pinsker's inequality,
\begin{align*}
\bbP_x\left(\tau < n \text{ and } \sum_{t=1}^{\tau+1} \Delta(\pi_t, x) \geq \frac{2^{-\lambda} \epsilon \Psi^\lambda_\star(\mu)}{\Delta_\star(\mu)^{\lambda - 1}}\right) 
> \frac{\delta}{m} - \sqrt{\epsilon/2}
= \frac{\delta}{2m}\,.
\end{align*}
Therefore,
\begin{align*}
\Reg_n^\star \geq \Reg_n(x) 
= \E_{\bbP_x}\left[\sum_{t=1}^n \Delta(\pi_t, x) \right] 
> \frac{2^{-\lambda}\epsilon \Psi^\lambda_\star(\mu)}{\Delta_\star(\mu)^{\lambda - 1}} \cdot \frac{\delta}{2m}
= \frac{2^{-\lambda-2} \varphi^3}{m^3\left(1 + \frac{1}{m}\right)^3} \cdot \frac{\Psi^\lambda_\star(\mu)}{\Delta_\star(\mu)^{\lambda - 1}}  \,.
\end{align*}

\paragraph{Step 4: Balancing the terms}
The results in the previous steps show that 
\begin{align*}
\Reg_n^\star \geq \const_{\cL\cS} \min\left(n \Delta_\star(\mu), \, \frac{\Psi^\lambda_\star(\mu)}{\Delta_\star(\mu)^{\lambda - 1}} \right)\,.
\end{align*}
Since this holds for any horizon $n$, let us choose $n = \floor{\Psi^\lambda_\star(\mu) / \Delta_\star(\mu)^\lambda}$.
Naive bounding and rearranging shows that for suitably large $\Psi^\lambda_\star(\mu)$,
\begin{align*}
\Reg_n(x) \geq \const_{\cL\cS} n^{1 - \frac{1}{\lambda}} \left(\Psi^\lambda_\star(\mu)\right)^{\frac{1}{\lambda}} \,.
\end{align*}
Since $\lambda < \lambda_\star$, $\Psi_\star^\lambda = \infty$ and so 
there exists a sequence $(\mu_t)_{t=1}^\infty$ in $\sP_m(\cK)$ with $\lim_{t\to\infty} \Psi^\lambda_\star(\mu_t) = \infty$.
Because $\Delta_\star(\mu) \in [0,1]$ for all $\mu$, the horizon as defined above is always at least $\Omega(\Psi^\lambda_\star(\mu))$.
Taking logarithms and a limit shows that $\limsup_{n\to\infty} \log(\Reg_n^\star) / \log(n) \geq 1 - 1/\lambda$ for all $\lambda < \lambda_\star$.
Taking the limit as $\lambda$ tends to $\lambda_\star$ from below completes the proof.
\end{proof}

\section{Cell Structure for Rustichini's Regret}\label{sec:cells}

Before the proof of Lemma~\ref{lem:key} for the Rustichini setting we need to explore the combinatorial and geometric structure of Rustichini's regret.
These concepts are needed anyway to understand why Rustichini games do not admit the same classification as finite partial monitoring games and to understand
our example in \cref{sec:exotic}.
Assume for this section that we are in the Rustichini setting so that $\V$ and $\V_\star$ are 
defined as in \cref{eq:rustichini} and $\cZ$ is finite and $\cK$ is the probability simplex
over $\cZ$.

Recall that $\V_\star(x) = \min_{\alpha \in [m]} \Vm(x)_\alpha$ with $\Vm : \cK \to \R^m$ linear.
For $\alpha \in [m]$, the cell of $\alpha$ is the convex polytope $P_\alpha = \{x \in \cK : \V_\star(x) = \Vm(x)_\alpha\}$.
The set $\cF = \cup_{\alpha \in [m]} \faces(P_\alpha)$ is a polyhedral complex (see \cref{sec:poly}) 
and corresponds to the cell decomposition in standard finite partial monitoring games.
A policy $\pi \in \Delta_k$ is optimal at $x \in \cK$ if $\Delta(\pi, x) = 0$ and the set of such policies is
denoted by $\cN(x) = \{\pi \in \Delta_k : \Delta(\pi, x) = 0\}$.
Note that $\cN(x)$ is nonempty by the definition of $\Delta$.
The following lemma shows that $\cN(x)$ only depends on the face $F$ such that $x \in \ri(F)$. The proof is given in \cref{app:tech}.

\begin{lemma}\label{lem:nash}
Suppose that $F \in \cF$ and $x,y \in \ri(F)$. Then $\cN(x) = \cN(y)$.
\end{lemma}

Based on Lemma~\ref{lem:nash}, it makes sense to define $\cN(F) = \cN(x)$ where $x \in \ri(F)$ is any point in the relative interior of $F$.
Note that because $x \mapsto \Delta(\pi, x)$ is continuous, it follows immediately that $\Delta(\pi, x) = 0$ for all $x \in F$ and $\pi \in \cN(F)$.
Continuity of $\Delta$ also implies that $\cN(F)$ is closed for all faces $F \in \cF$.
Further, $\pi \mapsto \Delta(\pi, x)$ is convex and piecewise linear for any $x$, so $\cN(F)$ is a convex polytope.
An essential difference between the cell decomposition in the standard setting and the Rustichini setting is that $\cN(P_\alpha)$ and $\cN(P_\beta)$ are not
necessary disjoint for distinct $\alpha, \beta \in [m]$.
The importance of this difference will become apparent in the next section.

\section{Exotic Rustichini Games}\label{sec:exotic}
\cref{thm:main} shows that the minimax regret of all finite Rustichini games is characterised by the generalised information ratio.
There has been some speculation about whether or not Rustichini games have a minimax regret in $\{0, \Theta(n^{1/2}), \Theta(n^{2/3})\}$, like
standard finite partial monitoring games except for the absence of hopeless games.
The theorem of this section shows that this is not the case.
The fundamental difference between Rustichini games and the standard setting is that policies can be optimal on multiple cells. 
In partial monitoring the regret is determined by the ability of the learner to estimate in which cell the adversary is playing, while
in Rustichini games there is a richer combinatorial structure. 
The example in the analysis below illustrates some of the challenges.

\begin{theorem}
There exists a finite Rustichini game such that $\limsup_{n\to\infty} \log(\Reg_n^\star) / \log(n) = 4/7$.
\end{theorem}

\begin{proof}[\textsc{sketch}]
The game we use to exhibit the desired minimax regret is written below in matrix form.
The construction is not exactly trivial and the important features of the game cannot be absorbed by staring at the matrices.
\begin{align*}
\cL &=
{\scriptsize
\left[
\begin{matrix}
1 & 0 & 0 & 0 & 1 & 1 & 0 & 0 \\
1 & 0 & 0 & 0 & 1 & 0 & 0 & 1 \\
1 & 1 & 0 & 0 & 1 & 0 & 0 & 0 \\
1 & 0 & 0 & 1 & 1 & 0 & 0 & 0 \\
1 & 1/4 & 0 & 1/4 & 1 & 1/4 & 0 & 1/4 \\
1 & 1/4 & 0 & 1/4 & 1 & 1/4 & 0 & 1/4 \\
1 & 1/4 & 0 & 1/4 & 1 & 1/4 & 0 & 1/4 \\
2 & 2 & 2 & 2 & 2 & 2 & 2 & 2 \\
2 & 2 & 2 & 2 & 2 & 2 & 2 & 2
\end{matrix}
\right]
}
&
\cS &=
{\scriptsize
\left[
\begin{matrix}
0 & 0 & 0 & 0 & 0 & 0 & 0 & 0 \\
0 & 0 & 0 & 0 & 0 & 0 & 0 & 0 \\
0 & 0 & 0 & 0 & 0 & 0 & 0 & 0 \\
0 & 0 & 0 & 0 & 0 & 0 & 0 & 0 \\
1 & 1 & 0 & 0 & 0 & 0 & 0 & 0 \\
0 & 1 & 1 & 0 & 0 & 0 & 0 & 0 \\ 
1 & 1 & 1 & 1 & 0 & 0 & 0 & 0 \\
0 & 0 & 0 & 0 & 1 & 1 & 0 & 0 \\
0 & 0 & 0 & 0 & 0 & 1 & 1 & 0 
\end{matrix}
\right] \,.
}
\end{align*}
Before diving into details, let us outline the intuitive role of the actions.
For this game, the set of optimal policies $\cN(x)$ for a given $x \in \cK$ are the convex combinations of the deterministic policies that are optimal at $x$. 
Note that is not true in general.
At every point $x$ one of the actions in $\{1,2,3,4\}$ are optimal. Actions $\{5,6,7\}$ are optimal on a low dimensional subset while actions $\{8,9\}$
are never optimal. Their only role is exploration and the regret when playing them is always large.


\paragraph{Loss structure}
The domain $\cK$ of the adversary is the probability simplex over the 8 outcomes, which is 7-dimensional and not easy to represent. Fortunately, the regret function
can be represented via a low dimensional parameterisation. A little notation is needed for this.
Let $u(x) = x_4 - x_2$ and $v(x) = x_8 - x_6$. 
A tedious calculation shows that
\begin{align*}
\Delta(\pi, x) \stackrel{\times}{=} \pi(\{8,9\}) +  
\begin{cases}
|u(x)|\pi(\{4,5,6,7\}) + |v(x)|\pi(\{2,5,6,7\}) & \text{if } u(x) \geq 0 \text{ and } v(x) \geq 0 \\
|u(x)|\pi(\{4,5,6,7\}) + |v(x)|\pi(\{1,5,6,7\}) & \text{if } u(x) \geq 0 \text{ and } v(x) \leq 0 \\
|u(x)|\pi(\{3,5,6,7\}) + |v(x)|\pi(\{2,5,6,7\}) & \text{if } u(x) \leq 0 \text{ and } v(x) \geq 0 \\
|u(x)|\pi(\{3,5,6,7\}) + |v(x)|\pi(\{1,5,6,7\}) & \text{if } u(x) \leq 0 \text{ and } v(x) \leq 0 \,. 
\end{cases}
\end{align*}
where $f \stackrel{\times}{=} g$ indicates that $f(x) \leq \const g(x) \leq \const f(x)$ for all $x$.
Notice this means that the deterministic policy that plays action $a = 1$ is optimal if $v(x) \geq 0$, while $a = 2$ is optimal if $v(x) \leq 0$
and $a = 3$ is optimal if $u(x) \geq 0$ and $a = 4$ is optimal if $u(x) \leq 0$.
Therefore the learner can identify an optimal policy by learning \textit{either} the sign of $u(x)$ or $v(x)$.

\paragraph{Signal structure}
Actions $\{1,2,3,4\}$ are entirely uninformative. 
Actions $\{5,6,7\}$ can be used to estimate $u(x)$.
By playing all actions the learner can estimate both $u(x)$ and $v(x)$.

\paragraph{Bounding the regret}
By the lower bound construction in \cref{thm:main} it suffices to consider only stochastic adversaries where $x_t = x$ for some unknown $x \in \cK$ for all $t$.
The learner wants $\Reg_n = O(n^{4/7})$ and the adversary wants $\Reg_n = \Omega(n^{4/7})$.
Clearly the learner can play all actions $\Theta(n^{4/7})$ times and use the information gained to estimate $u(x)$ and $v(x)$
to accuracy $O(n^{-2/7})$.
If the adversary chose $x$ such that $|u(x)| + |v(x)| = \Omega(n^{-2/7})$, then the learner can identify the sign of either $u(x)$ or $v(x)$ and
subsequently play optimally. 
On the other hand, if $|u(x)| + |v(x)| = O(n^{-2/7})$, then the learner can play actions $\{5,6,7\}$ $O(n^{6/7})$ times and learn $u(x)$ to accuracy $O(n^{-3/7})$.
Subsequently if $|u(x)| = \Omega(n^{-3/7})$, then the learner has enough information to know the sign of $u(x)$ and play optimally. On the other hand,
if $|u(x)| = O(n^{-3/7})$, then the learner plays either action in $\{3,4\}$ and suffers regret $O(n^{4/7})$ as required.
The adversary can force this by randomising $v(x) \in \{n^{-2/7}, -n^{-2/7}\}$ and $u(x) \in \{n^{-3/7}, -n^{-3/7}\}$.
Note the power of our lower bound construction in \cref{thm:main}. Except for subpolynomial factors we can prove a minimax adversarial regret upper bound by
analysing what is essentially explore-then-commit in the stochastic setting.
\end{proof}

\section{Discussion}

For infinite partial monitoring games with the standard regret it was not known what rates the minimax regret could have.
We show in Appendix~\ref{sec:standard-exotic} that for all $\lambda \in (2, \infty)$ there exists a game with $\lambda_\star = \lambda$.
The Rustichini regret and the standard regret usually coincide for games where $\cZ$ is finite and for which the standard minimax regret is not linear.
The cases where they do not are a subset of the so-called degenerate games. The details are explained in \cref{sec:compare}.

\paragraph{Classification of finite Rustichini games}
We believe that for finite Rustichini games 
\begin{align*}
\lim_{n\to\infty} \frac{\log(\Reg_n^\star)}{\log(n)} \in \{0\} \cup \{1/2\} \cup \left\{\frac{2^i}{2^{i+1}-1} : i \in 1,2,3,\ldots \right\}
\end{align*}
Games with this regret can be constructed by generalising the example in \cref{sec:exotic} to higher dimensions.
Proving that this kind of construction is essentially the only case remains an interesting challenge.

\paragraph{Other questions}
At the moment we do not know how to obtain high probability bounds in terms of the information ratio, which would be at least one route towards handling non-oblivious
adversaries. Our analysis shows that the minimax information ratio with respect to the logarithmic barrier characterises the regret. 
An interesting question is whether or not the same result holds when the logarithmic barrier is replaced by the negentropy, when the information gain would correspond to the
usual mutual information.
The obstruction in our current analysis is essentially \cref{eq:kl-rel1}, which does not hold anymore.
Even if the minimax information ratio with the negentropy does not yield a classification theorem, some version of exponential weights may still have good results.
\cite{DVX21} hints that this may sometimes be possible.
Understanding this may be important because the logarithmic barrier 
has a serious limitation that the diameter of its potential is linear in the dimension, which makes it poorly suited for high-dimensional structured problems.


\ifconf
\else
\bibliographystyle{plainnat}
\fi
\bibliography{all}

\begin{thebibliography}{23}
\providecommand{\natexlab}[1]{#1}
\providecommand{\url}[1]{\texttt{#1}}
\expandafter\ifx\csname urlstyle\endcsname\relax
  \providecommand{\doi}[1]{doi: #1}\else
  \providecommand{\doi}{doi: \begingroup \urlstyle{rm}\Url}\fi

\bibitem[Antos et~al.(2013)Antos, Bart{\'o}k, P{\'a}l, and
  Szepesv{\'a}ri]{ABPS13}
A.~Antos, G.~Bart{\'o}k, D.~P{\'a}l, and Cs. Szepesv{\'a}ri.
\newblock Toward a classification of finite partial-monitoring games.
\newblock \emph{Theoretical Computer Science}, 473:\penalty0 77--99, 2013.

\bibitem[Bart{\'o}k et~al.(2014)Bart{\'o}k, Foster, P{\'a}l, Rakhlin, and
  Szepesv{\'a}ri]{BFPRS14}
G.~Bart{\'o}k, D.~P. Foster, D.~P{\'a}l, A.~Rakhlin, and Cs. Szepesv{\'a}ri.
\newblock Partial monitoring---classification, regret bounds, and algorithms.
\newblock \emph{Mathematics of Operations Research}, 39\penalty0 (4):\penalty0
  967--997, 2014.

\bibitem[Blackwell(1956)]{Bla56}
D.~Blackwell.
\newblock An analog of the minimax theorem for vector payoffs.
\newblock \emph{Pacific Journal of Mathematics}, 6\penalty0 (1):\penalty0 1--8,
  1956.

\bibitem[Cesa-Bianchi et~al.(2006)Cesa-Bianchi, Lugosi, and Stoltz]{CBLuSt06}
N.~Cesa-Bianchi, G.~Lugosi, and G.~Stoltz.
\newblock Regret minimization under partial monitoring.
\newblock \emph{Mathematics of Operations Research}, 31:\penalty0 562--580,
  2006.

\bibitem[Devraj et~al.(2021)Devraj, Roy, and Xu]{DVX21}
A.~Devraj, B.~Van Roy, and K.~Xu.
\newblock A bit better? quantifying information for bandit learning.
\newblock \emph{arXiv preprint arXiv:2102.09488}, 2021.

\bibitem[Foster and Rakhlin(2012)]{FR12}
D.~Foster and A.~Rakhlin.
\newblock No internal regret via neighborhood watch.
\newblock In \emph{Proceedings of the 15th International Conference on
  Artificial Intelligence and Statistics}, pages 382--390, La Palma, Canary
  Islands, 2012. JMLR.org.

\bibitem[Foster et~al.(2021)Foster, Kakade, Qian, and Rakhlin]{FKQR21}
D.~Foster, S.~Kakade, J.~Qian, and A.~Rakhlin.
\newblock The statistical complexity of interactive decision making.
\newblock \emph{arXiv preprint arXiv:2112.13487}, 2021.

\bibitem[Kirschner et~al.(2021)Kirschner, Lattimore, Vernade, and
  Szepesv{\'a}ri]{KLVS21}
J.~Kirschner, T.~Lattimore, C.~Vernade, and Cs. Szepesv{\'a}ri.
\newblock Asymptotically optimal information-directed sampling.
\newblock In Mikhail Belkin and Samory Kpotufe, editors, \emph{Proceedings of
  Thirty Fourth Conference on Learning Theory}, volume 134 of \emph{Proceedings
  of Machine Learning Research}, pages 2777--2821. PMLR, 2021.

\bibitem[Kwon and Perchet(2017)]{KP17}
J.~Kwon and V.~Perchet.
\newblock Online learning and blackwell approachability with partial
  monitoring: optimal convergence rates.
\newblock In \emph{Artificial Intelligence and Statistics}, pages 604--613.
  PMLR, 2017.

\bibitem[Lattimore and Gy{\"o}rgy(2020)]{LG20}
T.~Lattimore and A.~Gy{\"o}rgy.
\newblock Mirror descent and the information ratio.
\newblock \emph{arXiv preprint arXiv:2006.00475}, 2020.

\bibitem[Lattimore and Szepesv{\'a}ri(2019{\natexlab{a}})]{LS18pm}
T.~Lattimore and Cs. Szepesv{\'a}ri.
\newblock Cleaning up the neighbourhood: A full classification for adversarial
  partial monitoring.
\newblock In \emph{Proceedings of the 30th International Conference on
  Algorithmic Learning Theory}, 2019{\natexlab{a}}.

\bibitem[Lattimore and Szepesv{\'a}ri(2019{\natexlab{b}})]{LS19pminfo}
T.~Lattimore and Cs. Szepesv{\'a}ri.
\newblock An information-theoretic approach to minimax regret in partial
  monitoring.
\newblock In \emph{Proceedings of the 32nd Conference on Learning Theory},
  pages 2111--2139, Phoenix, USA, 2019{\natexlab{b}}. PMLR.

\bibitem[Lattimore and Szepesv{\'a}ri(2019{\natexlab{c}})]{LS19pmsimple}
T.~Lattimore and Cs. Szepesv{\'a}ri.
\newblock Exploration by optimisation in partial monitoring.
\newblock arXiv:1907.05772, 2019{\natexlab{c}}.

\bibitem[Lattimore and Szepesv\'{a}ri(2020)]{LS20bandit-book}
T.~Lattimore and Cs. Szepesv\'{a}ri.
\newblock \emph{Bandit Algorithms}.
\newblock Cambridge University Press, 2020.

\bibitem[Lugosi et~al.(2008)Lugosi, Mannor, and Stoltz]{LMS08}
G.~Lugosi, S.~Mannor, and G.~Stoltz.
\newblock Strategies for prediction under imperfect monitoring.
\newblock \emph{Mathematics of Operations Research}, 33\penalty0 (3):\penalty0
  513--528, 2008.

\bibitem[Mannor and Shamir(2011)]{MS11}
S.~Mannor and O.~Shamir.
\newblock From bandits to experts: On the value of side-observations.
\newblock In \emph{Advances in Neural Information Processing Systems}, pages
  684--692. Curran Associates, Inc., 2011.

\bibitem[Mannor et~al.(2014)Mannor, Perchet, and Stoltz]{MPS14}
S.~Mannor, V.~Perchet, and G.~Stoltz.
\newblock Set-valued approachability and online learning with partial
  monitoring.
\newblock \emph{The Journal of Machine Learning Research}, 15\penalty0
  (1):\penalty0 3247--3295, 2014.

\bibitem[Perchet(2011{\natexlab{a}})]{Per11}
V.~Perchet.
\newblock Internal regret with partial monitoring: Calibration-based optimal
  algorithms.
\newblock \emph{Journal of Machine Learning Research}, 12\penalty0 (6),
  2011{\natexlab{a}}.

\bibitem[Perchet(2011{\natexlab{b}})]{Per11b}
V.~Perchet.
\newblock Approachability of convex sets in games with partial monitoring.
\newblock \emph{Journal of Optimization Theory and Applications}, 149\penalty0
  (3):\penalty0 665--677, 2011{\natexlab{b}}.

\bibitem[Piccolboni and Schindelhauer(2001)]{PS01}
A.~Piccolboni and C.~Schindelhauer.
\newblock Discrete prediction games with arbitrary feedback and loss.
\newblock In \emph{Computational Learning Theory}, pages 208--223. Springer,
  2001.

\bibitem[Russo and {Van Roy}(2014)]{RV14}
D.~Russo and B.~{Van Roy}.
\newblock Learning to optimize via information-directed sampling.
\newblock In \emph{Advances in Neural Information Processing Systems}, pages
  1583--1591. Curran Associates, Inc., 2014.

\bibitem[Rustichini(1999)]{Rus99}
A.~Rustichini.
\newblock Minimizing regret: The general case.
\newblock \emph{Games and Economic Behavior}, 29\penalty0 (1):\penalty0
  224--243, 1999.

\bibitem[Zimmert and Lattimore(2019)]{ZL19}
J.~Zimmert and T.~Lattimore.
\newblock Connections between mirror descent, thompson sampling and the
  information ratio.
\newblock In \emph{Advances in Neural Information Processing Systems}, pages
  11973--11982. Curran Associates, Inc., 2019.

\end{thebibliography}

\appendix

\section{Polyhedral complexes and triangulations}\label{sec:poly}

Here we collect some definitions and basic facts about polyhedral complexes and the existence of triangulations.
Throughout the section, $\cK$ is a compact convex subset of $\R^d$.
Given $x, y \in \R^d$, let $[x,y] = \{(1 - \alpha) x + \alpha y : \alpha \in [0,1]\}$ denote the closed chord connecting
$x$ and $y$ and $(x, y) = \{(1 - \alpha) x + \alpha y : \alpha \in (0, 1)\}$ the open chord.
Given a polytope $P$, a set $F \subset P$ is a face of $P$ if for all $x, y \in P$ with $(x, y) \cap F \neq \emptyset$, $[x,y] \subset F$.
Note that $P$ and the empty set are faces of $P$.
Remember that the set of all faces is denoted by $\faces(P)$.

\begin{definition}
A convex polyhedral complex on $\cK$ is a finite set of closed convex polytopes $\cP$ such that: 
\begin{enumerate}[noitemsep,nolistsep]
\item[\textit{(a)}] $\faces(P) \subset \cP$ for all $P \in \cP$; and
\item[\textit{(b)}] $P \cap Q \in \cP$ for all $P, Q \in \cP$; and
\item[\textit{(c)}] $\cK = \cup_{P \in \cP} P$.
\end{enumerate}
\end{definition}

\begin{fact}\label{fact:poly-ri}
Given any $x \in \cK$ and convex polyhedral complex $\cP$ on $\cK$, there exists a unique $P \in \cP$ such that $x \in \ri(P)$.
\end{fact}

A polyhedral complex $\cP$ is called a simplicial complex if $P$ is a simplex for all $P \in \cP$.

\begin{fact}\label{fact:poly-tri}
Given any continuous piecewise linear function $f : \cK \to \R^s$, there exists a simplicial complex $\cP$ on $\cK$ such that
the restriction $f_{|P}$ is linear for all $P \in \cP$.
\end{fact}

\section{Technical lemmas}\label{app:tech}

\begin{proof}[\textsc{Lemma~\ref{lem:V}}]
We start by proving the claims in the standard setting, which means that
that $\V$ and $\V_\star$ are defined by \cref{eq:standard}.
Since $x \mapsto \cL(\pi, x) = \int_\cZ \cL(\pi, z) \d{x}(z)$ is linear, it follows that $\V_\star$ is the minimum of $k$ linear functions and hence is
concave and piecewise linear with at most $k$ pieces.
Parts (b) and (c) in the standard setting are trivial, since $\V$ is bilinear.
Moving on to the Rustichini setting. Since $\cZ$ is finite, $\cK$ corresponds to a probability simplex over $\cZ$.
Since $\cZ$ is finite we assume without loss of generality that $\Sigma$ is finite and $\Sigma = [|\Sigma|]$.
Notice that $\V_\star$ can be written as a linear program and its dual is
\begin{align*}
\V_\star(x) 
&= \min_{\pi \geq \zeros} \min_{\lambda} \ip{\lambda, Sx} 
\quad \text{subject to } \pi^\top \cL - S^\top \lambda \leq \zeros \text{ and } \ip{\pi, \ones} = 1\,.
\end{align*}
where $S \in \{0,1\}^{a|\Sigma| \times d}$ is the matrix with $S_{a\sigma, z} = \sind(\cS(a, z) = \sigma)$.
And again, $\V_\star$ can be written as the minimum of finitely many linear functions. 
The results for $\V$ follow along the same lines.
Let $\pi \in \Delta_k$ be fixed and consider the mapping $x \mapsto \V(\pi, x)$, which by duality satisfies
\begin{align*}
\V(\pi, x) = \min_{\lambda} \ip{\lambda, Sx} \quad \text{subject to } \pi^\top \cL - S^\top \lambda \leq \zeros \text{ and } \ip{\pi, \ones} = 1 \,.
\end{align*}
As above, this is the minimum of finitely many linear functions and hence concave and piecewise linear.
For the last part, let $x \in \cK$ be fixed, then
\begin{align*}
\V(\pi, x) = \max_{y : y \rel x} \cL(\pi, y) \,.
\end{align*}
Since $\{y \in \cK : y \rel x\}$ is a polytope, the maximum is achieved on an extreme point of this polytope of which there are finitely many. Hence $\pi \mapsto \V(\pi, x)$
is the maximum of finitely many linear functions and hence is convex and piecewise linear.
\end{proof}

\begin{proof}[\textsc{Lemma~\ref{lem:psi-mix}}]
Let $\rho$ be a policy such that $\Delta_\star(\mu) = \Delta(\rho, \mu)$, which exists by Lemma~\ref{lem:exist}. Then,
\begin{align*}
\Psi^\lambda_\star(\mu) 
&\leq \min_{p \in [0,1]} \frac{\left(p \Delta(\pi, \mu) + (1 - p) \Delta(\rho, \mu)\right)^\lambda}{p \I(\pi, \mu)} 
\leq \frac{2^\lambda \Delta_\star(\mu)^{\lambda - 1} \Delta(\pi, \mu)}{\I(\pi, \mu)}\,, 
\end{align*}
where the first inequality follows from convexity of $\pi \mapsto \Delta(\pi, \mu)$ and because $\I(p \pi + (1 - p) \rho, \mu) \geq p \I(\pi, \mu)$. The second inequality follows 
by choosing $p = \Delta(\rho, \mu) / \Delta(\pi, \mu)$, which is in $[0,1]$ by the definition of $\rho$.
\end{proof}

\begin{proof}[\textsc{Lemma~\ref{lem:pinsker}}]
Let $A \in \cF_t$ and
$\tau = \min\{s \leq t : \sum_{u=1}^{s+1} \I(\pi_u, \mu) \geq \epsilon\}$ with the minimum of the empty set defined to be $t+1$. 
Because $\I(\pi_{u+1}, \mu)$ is $\cF_u$-measurable, $\tau$ is a stopping time adapted to $(\cF_s)_{s=1}^t$.
Therefore,
\begin{align*}
\bbP^t_x(A) 
&\leq \bbP^t_x(\tau < t) + \bbP^t_x(A \cap \{\tau \geq t\}) \\
&\leq \bbP^t(\tau < t) + 2\sqrt{\frac{\epsilon}{2}} + \bbP^t(A \cap \{\tau \geq t\}) \\
&\leq \bbP^t(\tau < t) + 2\sqrt{\frac{\epsilon}{2}} + \bbP^t(A) \\ 
&= \bbP^t\left(\sum_{s=1}^t \I(\pi_s, \mu) \geq \epsilon\right) + \sqrt{2\epsilon} + \bbP^t(A) \,,
\end{align*}
where in the second inequality we used Pinsker's inequality and \cref{eq:kl-rel1} and
the fact that both $\{\tau < t\} \in \cF_\tau$ and $A \cap \{\tau \geq t\} \in \cF_\tau$.
The other direction follows along the same lines.
\end{proof}

\begin{proof}[\textsc{Lemma~\ref{lem:nash}}]
Let $\pi \in \cN(y)$.
Since $x, y \in \ri(F)$, there exists an $\epsilon > 0$ such that $y + \epsilon(y - x) \in F$.
Since $\V_\star$ is linear on $F$, $z \mapsto \Delta(\pi, z)$ is concave on $F$.
Therefore,
\begin{align*}
0 \leq \Delta(\pi, y + \epsilon(y - x)) \leq \epsilon(\Delta(\pi, y) - \Delta(\pi, x)) = -\epsilon \Delta(\pi, x)\,,
\end{align*}
which implies that $\Delta(\pi, x) = 0$.
Therefore $\cN(x) \subset \cN(y)$. A symmetric argument shows that $\cN(y) \subset \cN(x)$ and hence $\cN(x) = \cN(y)$.
\end{proof}

\begin{lemma}\label{lem:sep}
Suppose that $(C_i)_{i=1}^k$ is a finite collection of compact sets in $\R^d$ with $\bigcap_{i=1}^k C_i = \emptyset$.
Then there exists an $\epsilon > 0$ such that 
\begin{align*}
\inf_{x \in \R^d} \max_{i \in [k]} \dist(x, C_i) > \epsilon\,.
\end{align*}
\end{lemma}

\begin{proof}
Suppose not. There exists a sequence $(x_t)_{t=1}^\infty$ such that $\lim_{n\to\infty} \max_{i \in [k]} \dist(x_t, C_i) = 0$ for all $i \in [k]$.
Clearly the sequence $(x_t)$ is bounded and hence has a convergent subsequence. Let $x$ be an accumulation point $(x_t)$. Then, since $C_i$ is closed, $d(x, C_i) = 0$ for 
all $i$ and hence $x \in \bigcap_{i \in [k]} C_i$, which is a contradiction.
\end{proof}

\begin{lemma}\label{lem:exist}
For all $\lambda > 1$ and $\mu \in \sP(\cK)$ there exists a policy $\pi \in \Delta_k$ such that $\Psi^\lambda(\pi, \mu) = \Psi^\lambda_\star(\mu)$.
\end{lemma}

\begin{proof}
There is nothing to prove if $\Psi^\lambda_\star(\mu) = \infty$, so assume this is not the case.
If there exists a $\pi$ such that $\Delta(\pi, \mu) = 0$. Then $\Psi^\lambda_\star(\mu) = \Psi^\lambda(\pi, \mu) = 0$.
Suppose now that $\Delta(\pi, \mu) > 0$ for all $\pi$.
Since $\pi \mapsto \Delta(\pi, \mu)$ is continuous and $\Delta_k$ is compact, $\min_{\pi \in \Delta_k} \Delta(\pi, \mu) > 0$.
Let $\pi_1,\pi_2,\ldots$ be a sequence of policies with $\lim_{t \to \infty} \Psi^\lambda(\pi, \mu) = \Psi^\lambda_\star(\mu)$.
By compactness this sequence has a convergent subsequence converging to some policy $\pi_\star$.
Clearly there exists an $a \in \spt(\pi)$ with $\KL(\cS(a, \mu), \cS(a, x)) > 0$. Therefore $\pi \mapsto \Psi^\lambda(\pi, \mu)$ is continuous in a neighbourhood of $\pi_\star$ and
the result follows.
\end{proof}

\begin{lemma}\label{lem:dist}
Suppose that $F \subset \R^d$ is a polytope. Then $x \mapsto \dist(x, F)$ is piecewise linear.
\end{lemma}

\begin{proof}
Suppose that $F = \{x : Ax \leq b\}$ for suitable $A \in \R^{s\times d}$ and $b \in \R^s$ and let $S \in \R^{2^d \times d}$ be the matrix with rows formed from all possible elements of $\{-1,1\}^d$. Then,
\begin{align*}
\dist(x, F) 
&= \min_{\substack{y \in \R^d \\ z \in \R}} \left\{ z : Ay \leq b, S(y-x) \leq z \ones \right\} \\
\tag{duality} &= \max_{\substack{\zeros \leq \lambda \in \R^s \\ \zeros \leq \varphi \in \R^{2^d}}} \left\{\ip{\varphi, Sx} - \ip{\lambda, b} : A^\top \lambda = S^\top \varphi, \ip{\varphi, \ones} = 1\right\}\,,  
\end{align*}
which can be written as the maximum of finitely many linear functions and hence is convex and piecewise linear.
\end{proof}

\section{Proof of Lemma~\ref{lem:key} [\textsc{rustichini setting}]}\label{sec:lem:key}

The mission in this section is to prove Lemma~\ref{lem:key} in the Rustichini setting.
The notation in \cref{sec:cells} will be used, so that section should be read first.

\begin{lemma}\label{lem:nash-lip}
There exists a game-dependent constant $\lip_\Delta$ such that for all nonempty $F \in \cF$,
\begin{align*}
\max_{\pi \in \cN(F)} \Delta(\pi, x) \leq \lip_\Delta \dist(x, F)\,.
\end{align*}
\end{lemma}

\begin{proof}
Let $F \in \cF$ be nonempty and recall from \cref{sec:cells} that $\cN(F)$ is a polytope. 
By Lemma~\ref{lem:V}, $\pi \mapsto \Delta(\pi, x)$ is convex for all $x \in \cK$ and $x \mapsto \Delta(\pi, x)$ is piecewise linear for all $\pi \in \Delta_k$.
Since $\cK$ is compact, there can be only finitely many pieces.
Convex functions on polytopes are maximised on their extreme points and so
\begin{align*}
\max_{\pi \in \cN(F)} \Delta(\pi, x) = \max_{\pi \in \ext(\cN(F))} \Delta(\pi, x)\,.
\end{align*}
Since $x \mapsto \Delta(\pi, x)$ is piecewise linear, for any $\pi \in \Delta_k$ there exists a constant $c_\pi \geq 0$ such that $x \mapsto \Delta(\pi, x)$ is $c_\pi$-Lipschitz with respect to 
$\norm{\cdot}_1$. Let $y \in F$ be such that $\dist(x, F) = \norm{x - y}_1$, which exists because $F$ is closed. Therefore, 
\begin{align*}
\max_{\pi \in \ext(\cN(F))} \Delta(\pi, x) \leq \max_{\pi \in \ext(\cN(F))} \Delta(\pi, y) + c_\pi \norm{x - y}_1 
= \dist(x, F) \max_{\pi \in \ext(\cN(F))} c_\pi \,,
\end{align*}
where we used the fact that for $\pi \in \ext(\cN(F))$ and $y \in F$, $\Delta(\pi, y) = 0$.
The result follows with $\lip_\Delta = \max_{F \in \cF : F \neq \emptyset} \max_{\pi \in \ext(\cN(F))} c_\pi$, which is finite because $\cF$ is finite and $\ext(\cN(F))$ is finite
for all $F \in \cF$.
\end{proof}

\begin{lemma}\label{lem:triangle}
Let $\epsilon > 0$, $\cG \subset \cF$ and $\Pi = \{\pi : \dist(\pi, \cN(F)) \geq \epsilon \text{ for all } F \in \cG\}$ and define $f, g : \cK \to \R$ by
\begin{align*}
f(x) &= \min_{\pi \in \Pi} \Delta(\pi, x) &
g(x) &= \dist\left(x, \cup_{F \in \cF \setminus \cG} F\right) \,.
\end{align*}
Then there exists a strictly positive constant $\const_{\epsilon}$ depending on the game and $\epsilon$ such that
\begin{align*}
f(x) \geq \const_{\epsilon} g(x) \,.
\end{align*}
\end{lemma}

\begin{proof}
If $\cG$ is empty, then $g(x) = 0$ for all $x$ and the result is immediate because $f$ is non-negative.
Assume for the remainder that $\cG$ is nonempty.
Note that $\Pi$ is can be written as a union of convex polytopes. Therefore,
\begin{align*}
f(x) = \min_{\pi \in \Pi} \max_{y : y \rel x} \cL(\pi, y) - \V_\star(x)\,,
\end{align*}
is the difference between two concave piecewise linear functions and hence is piecewise linear.
The function $g$ is also piecewise linear because
\begin{align*}
g(x) = \dist(x, \cup_{F \in \cF \setminus \cG} F) = \min_{F \in \cF \setminus \cG} \dist(x, F)\,,
\end{align*}
which is the minimum of finitely many piecewise linear functions (Lemma~\ref{lem:dist}) and hence piecewise linear.
Suppose that $F \in \cG$. 
By definition, if $x \in \ri(F)$, then $\Delta(\pi, x) = 0$ if and only if $\pi \in \cN(F)$.
Hence, because $\dist(\pi, \cN(F)) \geq \epsilon$ for all $\pi \in \Pi$, $f(x) > 0$ for all $x \in \ri(F)$.
Let $\cT$ be a simplicial complex of $\cK$ such that the restrictions $f_{|T}$ and $g_{|T}$ are linear for all $T \in \cT$, which exists by Fact~\ref{fact:poly-tri}. 
Given $T \in \cT$ and $t \in \ext(T)$, let $F \in \cF$ be such that $t \in \ri(F)$, which is unique by Fact~\ref{fact:poly-ri}. 
Suppose that $g(t) > 0$. Then $F \in \cG$, since otherwise $g(t) = \dist(t, \cup_{F \in \cF \setminus \cG} F) = 0$.
Therefore $f(t) > 0$. Hence,
\begin{align}
g(t) > 0 \implies f(t) > 0 \text{ for all }t \in \ext(T)\,.
\label{eq:abs}
\end{align}
Let $x \in T$ and $p \in \sP(\ext(T))$ be a distribution such that
$x = \sum_{t \in \ext(T)} p(t) t$. Combining \cref{eq:abs} with the facts that both $f$ and $g$ are linear on $T$ and non-negative,
\begin{align*}
g(x) 
\tag{$g_{|T}$ linear, $g \geq 0$} &= \sum_{t \in \ext(T) : g(t) > 0} p(t) g(t) \\
\tag{by \cref{eq:abs}} &= \sum_{t \in \ext(T) : g(t) > 0} p(t) f(t) \frac{g(t)}{f(t)} \\
\tag{$f,g \geq 0$} &\leq \left(\max_{t \in \ext(T) : g(t) > 0} \frac{g(t)}{f(t)}\right) \sum_{t \in \ext(T) : g(t) > 0} p(t) f(t) \\
\tag{$f \geq 0$} &\leq \left(\max_{t \in \ext(T) : g(t) > 0} \frac{g(t)}{f(t)}\right) \sum_{t \in \ext(T)} p(t) f(t) \\
&=\tag{$f_{|T}$ linear} \left(\max_{t \in \ext(T) : g(t) > 0} \frac{g(t)}{f(t)}\right) f(x)\,.
\end{align*}
Rearranging shows that
\begin{align*}
f(x) \geq \left(\min_{t \in \ext(T) : g(t) > 0} \frac{f(t)}{g(t)}\right) g(x)\,.
\end{align*}
The minimum is always positive for any $T$ by \cref{eq:abs}.
The result follows since there are only finitely many $T \in \cT$.
\end{proof}

\begin{proof}[\textsc{Lemma~\ref{lem:key}, Rustichini setting}]
We proceed in a few steps.

\paragraph{Step 1: Separation arguments}
Let $\epsilon > 0$ be a constant such that for any finite index set $\cI$ and set $(\cF_\alpha)_{\alpha \in \cI}$ of subsets of $\cF$ such that whenever
\begin{align*}
\bigcap_{\alpha \in \cI} \bigcup_{F \in \cF_\alpha} \cN(F) = \emptyset\,,
\end{align*}
it holds that 
\begin{align*}
\min_{\pi \in \Delta_k} \max_{\alpha \in \cI} \dist\left(\pi, \bigcup_{F \in \cF_\alpha} \cN(F)\right) > \epsilon\,.
\end{align*}
Such a constant exists by Lemma~\ref{lem:sep} and because there are only finitely many subsets of $\cF$ and because $\cN(F)$ is closed for all $F \in \cF$ and the union
of finitely many closed sets is closed.

\paragraph{Step 2: Separation arguments} 
The claim of the lemma is trivial if $\Delta_\star(\mu) = 0$, so for the remainder assume that $\Delta_\star(\mu) > 0$.
For $x \in \spt(\mu)$ let $\cF_x = \{F \in \cF : \dist(x, F) < \Delta_\star(\mu) / \lip_\Delta \}$, where $\lip_\Delta$ is the constant from Lemma~\ref{lem:nash-lip}.
Suppose that
\begin{align*}
\bigcap_{x \in \spt(\mu)} \bigcup_{F \in \cF_x} \cN(F) \neq \emptyset\,.
\end{align*}
Then there exists a policy $\pi$ such that for all $x \in \spt(\mu)$, $\pi \in \cup_{F \in \cF_x} \cN(F)$.
By Lemma~\ref{lem:nash-lip}, for this policy it holds for all $x \in \spt(\mu)$ that
\begin{align*}
\Delta(\pi, x) \leq \lip_\Delta \max_{F \in \cF_x} \dist(x, F) < \Delta_\star(\mu)\,,
\end{align*}
which contradicts the definition of $\Delta_\star(\mu)$ as the minimum of $\pi \mapsto \Delta(\pi, \mu)$.
Therefore,
\begin{align*}
\bigcap_{x \in \spt(\mu)} \bigcup_{F \in \cF_x} \cN(F) = \emptyset\,.
\end{align*}
By the definition of $\epsilon$ in the first step,
\begin{align*}
\min_{\pi \in \Delta_k} \max_{x \in \spt(\mu)} \dist\left(\pi, \bigcup_{F \in \cF_x} \cN(F)\right) > \epsilon\,.
\end{align*}

\paragraph{Step 3: Perturbations}
Let $\varphi = \epsilon / (2k)$ and $(\pi_x)_{x \in \spt(\mu)}$ be a collection of policies and $\pi$ be a policy such that 
$\norm{\pi - \pi_x}_\infty \leq \varphi$. Then $\norm{\pi - \pi_x}_1 \leq \epsilon/2$ for all $x \in \spt(\mu)$.
By the previous step there exists an $x \in \spt(\mu)$ such that
\begin{align*}
\dist\left(\pi, \bigcup_{F \in \cF_x} \cN(F)\right) > \epsilon\,.
\end{align*}
Therefore, by the triangle inequality,
\begin{align*}
\dist\left(\pi_x, \bigcup_{F \in \cF_x} \cN(F)\right) > \epsilon/2\,.
\end{align*}
By Lemma~\ref{lem:triangle}, 
\begin{align*}
\Delta(\pi_x, x) 
&\geq \const_{\epsilon/2,\cF_x} \dist\left(x, \bigcup_{F \in \cF \setminus \cF_x} F\right) \\
&= \const_{\epsilon/2,\cF_x} \min_{F \in \cF \setminus \cF_x} \dist(x, F) \\
&\geq \frac{\const_{\epsilon/2,\cF_x}}{\lip_\Delta} \Delta_\star(\mu)\,,
\end{align*}
where in the final inequality we used the definition of $\cF_x$.
The result follows because there are only finitely many possibles values of $\cF_x$.
\end{proof}

\section{Exotic standard games}\label{sec:standard-exotic}

Many classical partial monitoring problems have infinite latent spaces. Even the most elementary prediction with expert advice problem, for example, has 
infinite latent spaces because the losses are allowed to take real values.
\cref{thm:main} shows that the information ratio characterises the regret for games like this.
More interestingly, the information ratio also characterises the regret for more exotic games.
Here we show that in contrast to finite partial monitoring games, when the latent space is infinite, then the minimax regret
can have a continuum of rates.

\begin{theorem}
For any $q \in (1/2,1)$ there exists a partial monitoring game for which
\begin{align*}
\lim_{n\to\infty} \frac{\log(\Reg_n^\star)}{\log(n)} = q\,.
\end{align*}
\end{theorem}

Note that for $q \in \{0,1/2,1\}$ the above is already known even for finite partial monitoring games.

\begin{proof}
The construction we employ is based on a lower bound for finite partial monitoring by \cite{LS18pm} showing that the size of the signal set
can play a role in the minimax regret for some games.

\paragraph{Step 1: Constructing the game}
Let $p = 2 - 2q \in (0,1)$.
For $\alpha \in \N$ define a game $\cG_\alpha$ with $k = 2$ actions and $2\alpha - 1$ outcomes with loss and signal matrices given by
\begin{align*}
\cL_\alpha &=
\left[
\begin{matrix}
0     & \alpha^{-p} & \cdots & \alpha^{-p} & 0     \\
\alpha^{-p} & 0     & \cdots & 0     & \alpha^{-p}
\end{matrix}
\right] & 
\cS_\alpha &= 
\left[
\begin{matrix}
1 & 2 & 2 & 3 & 3 & \cdots & \alpha & \alpha \\
1 & 1 & 2 & 2 & 3 & \cdots & \alpha-1 & \alpha
\end{matrix}
\right] \,.
\end{align*}
This is the same game used by \cite{LS18pm} except that the losses have been scaled by a factor of $\alpha^{-p}$.
Let $g_\alpha(a, \sigma) = \alpha^{-p} (a + 2(\sigma - 1))^{a + 2\sigma - 1}$, which has the property that for all $i \in [2\alpha - 1]$,
\begin{align*}
\sum_{a \in [k]} g_\alpha(a, \cS_\alpha(a, i)) = \cL_\alpha(1, i) - \cL_\alpha(2, i) = \alpha^{-p} (-1)^i \,.
\end{align*}
In the language of finite partial monitoring, $g_\alpha$ is an estimation function that can be used to estimate the differences in losses between actions 
using only the observed signals.
Note that $\norm{g_\alpha}_\infty \leq 2\alpha^{1-p}$.
Next, let $\cG$ be the infinite game where the adversary chooses
an interleaving of outcomes from games $(\cG_\alpha)_{\alpha=1}^\infty$ and the learner gets to observe the signal
corresponding to its action and the game/outcome chosen by the adversary and also the identity of the game. Formally,
\begin{align*}
\cZ &= \cup_{\alpha \in \N} \{(\alpha, i) : i \in [2\alpha - 1]\} \\
\Sigma &= \cup_{\alpha \in \N} \{(\alpha, i) : i \in [\alpha]\} \\
\cL(a, (\alpha, i)) &= \cL_\alpha(a, i) \\
\cS(a, (\alpha, i)) &= \{\alpha, \cS_\alpha(a, i)\}\,.
\end{align*}

\paragraph{Step 2: Lower bound}
Let $\Reg_n^\star(\cG_\alpha)$ denote the minimax regret for the game $\cG_\alpha$.
\cite{LS18pm} showed that for any $\alpha$, 
\begin{align*}
\Reg_n^\star(\cG_\alpha) \geq \const \min(\alpha^{1-p} \sqrt{n}, n \alpha^{-p}) \,,
\end{align*}
where $\const > 0$ is a universal constant.
Obviously the adversary can choose to limit its choices to any particular game and hence
\begin{align*}
\Reg_n^\star(\cG) \geq \const \max_{\alpha \in \N} \min(\alpha^{1-p} \sqrt{n}, n \alpha^{-p})) \,.
\end{align*}
By optimising $\alpha = \ceil{\sqrt{n}}$ and passing to the limit we conclude that
\begin{align*}
\lim_{n\to\infty} \frac{\log(\Reg_n^\star(\cG))}{\log(n)} \geq 1 - p/2 = q\,.
\end{align*}

\paragraph{Step 3: Upper bound}
To obtain the upper bound we control the generalised information ratio.
Let $g(a, (\alpha, \sigma)) = g_\alpha(a, \sigma)$ where $g_\alpha$ is the function defined in the first step.
By the definitions, we have the following:
\begin{enumerate}[noitemsep,nolistsep]
\item [\textit{(a)}] $|g(a, (\alpha, \sigma))| \leq 2\alpha^{1-p}$ for all actions $a \in [k]$ and $(\alpha, \sigma) \in \Sigma$.
\item [\textit{(b)}] $\sum_{a \in [k]} g(a, \cS(a, z)) = \cL(2, z) - \cL(1, z)$ for all $z \in \cZ$.
\end{enumerate}
Because there are only two actions it suffices to bound the information ratio for priors $\mu \in \sP_2(\cK)$ so let
$x, y \in \cK$ be arbitrary and $\mu(x) = s$ and $\mu(y) = 1-s$ for some $s \in [0,1]$. We need to show that
\begin{align*}
\Psi^{2/p}_\star(\mu) = \min_{\pi \in \Delta_k} \frac{\Delta(\pi, \mu)^{2/p}}{\I(\pi, \mu)} \leq \const\,,
\end{align*}
where $\const$ is a universal constant that is independent of $\mu$.
Let $\pi \in \Delta_k$ be the distribution with $\pi(1) = s$ and $\pi(2) = 1 - s$ and $\cZ_{<\beta} = \{(\alpha, z) \in \cZ : \alpha < \beta\}$.
Then, letting $u = s x + (1 - s) y$, which satisfies $\cS(a, u) = \cS(a, \mu)$,
\begin{align*}
&\Delta(\pi, \mu)
= s\left(\cL(\pi, x) - \cL(1, x)\right) + (1 - s)\left(\cL(\pi, y) - \cL(2, y)\right) \\
&= s(1-s)\left(\cL(2, x) - \cL(1, x)\right) + s(1-s)\left(\cL(1, y) - \cL(2, y)\right) \\
&= s(1-s) \int_{\cZ} \left(\cL(2, z) - \cL(1, z)\right) x(\d{z}) + s(1-s)\int_{\cZ} \left(\cL(1, y) - \cL(2, y)\right) y(\d{z}) \\
&\leq \frac{\beta^{-p}}{2} + s(1-s) \int_{\cZ_{<\beta}} \left(\cL(2, z) - \cL(1, z)\right) x(\d{z}) + s(1-s)\int_{\cZ_{<\beta}} \left(\cL(1, y) - \cL(2, y)\right) y(\d{z}) \\
&= \frac{\beta^{-p}}{2} + s(1-s) \int_{\cZ_{<\beta}} \sum_{a \in [k]} g(a, \cS(a, z)) x(\d{z}) - s(1-s)\int_{\cZ_{<\beta}} \sum_{a \in [k]} g(a, \cS(a, z) y(\d{z}) \\
&= \frac{\beta^{-p}}{2} + s(1-s) \int_{\cZ_{<\beta}} \sum_{a \in [k]} g(a, \cS(a, z)) \left((x - u) + (u - y)\right)(\d{z}) \\
&\leq \frac{\beta^{-p}}{2} + \beta^{1-p} s(1-s) \sum_{a \in [k]} \left(\sqrt{\frac{\KL(\cS(a, \mu), \cS(a, x))}{2}} + \sqrt{\frac{\KL(\cS(a, \mu), \cS(a, y)}{2}}\right) \\ 
&\leq \frac{\beta^{-p}}{2} + \beta^{1-p} s(1-s) \sum_{a \in [k]} \sqrt{\KL(\cS(a, \mu), \cS(a, x)) + \KL(\cS(a, \mu), \cS(a, y))} \\ 
&= \frac{\beta^{-p}}{2} + \beta^{1-p} s(1-s) \sum_{a \in [k]} \sqrt{\I(a, \mu)} \\
&\leq \frac{\beta^{-p}}{2} + \beta^{1-p} \sqrt{\I(\pi, \mu)} \,. 
\end{align*}
By choosing $\beta = \ceil{\sqrt{1/\I(\pi, \mu)}}$, it follows that
\begin{align*}
\Delta(\pi, \mu) \leq \const \I(\pi, \mu)^{p/2}\,.
\end{align*}
Therefore, $\Psi^{2/p}_\star(\mu) \leq \const$ and
\begin{align*}
\Reg_n^\star \leq \const n^{1-p/2} \log(n)^{p/2} \,. 
\end{align*}
The result follows by passing to the limit.
\end{proof}

\section{Proof of \cref{thm:exp-by-opt}}\label{sec:thm:exp-by-opt}

Our proof combines and mirrors the proofs of Theorems 8 and 9 of \cite{LG20}.
There are only two novel aspects. 
One is that we use convexity of $\pi \mapsto \V(\pi, x)$ and concavity of $x \mapsto \V(\pi, x)$, which in the 
previous analysis were always linear. 
The second (marginally) novel part is to relate our definition of the information gain to the Bregman divergences used by \cite{LG20}, which
is really just a moderately tricky application of Bayes' law.
Recall that $F(p) = -\sum_{\alpha=1}^m \log(p_\alpha)$ is the logarithmic barrier on $\cD = \Delta_m \cap [1/n,1]^m$.

\paragraph{Step 1: Bounding the regret by the stability}
We start by bounding the regret in terms of $\Lambda^\star_{q\eta}$.
Using the definition of the regret and $\Lambda^\star_{q\eta}$ and letting $\epsilon_\cD = \frac{1}{n} \max_{x \in \cK} \norm{\Vm(x)}_\infty$,
\begin{align*}
\Reg_n 
&= \E\left[\sum_{t=1}^n \V(\pi_t, x_t) - n \V_\star\left(\frac{1}{n} \sum_{t=1}^n x_t\right)\right] \\
&= \E\left[\sum_{t=1}^n \V(\pi_t, x_t) - n \min_{p \in \Delta_m} \bip{p, \Vm\left(\frac{1}{n} \sum_{t=1}^n x_t\right)}\right] \\
&= \max_{p \in \Delta_m} \E\left[\sum_{t=1}^n \V(\pi_t, x_t) - \bip{p, \Vm(x_t)}\right] \\
&\leq \max_{p \in \cD} \E\left[\sum_{t=1}^n \V(\pi_t, x_t) - \bip{p, \Vm(x_t)}\right] + 2 n \epsilon_\cD  \,.
\end{align*}
By the standard analysis of mirror descent \citep[Theorem 28.4]{LS20bandit-book} with $\diam(\cD) = \max_{x,y \in \cD} F(x) - F(y)$ and $p \in \cD$, then
\begin{align*}
\sum_{t=1}^n \bip{q_t - p, \frac{\hat v_t(a_t, \sigma_t)}{\eta}} 
&\leq \frac{\diam(\cD)}{\eta} + \frac{1}{\eta} \sum_{t=1}^n \Breg_\star\left(\nabla F(q_t) - \frac{\eta \hat v(a_t, \sigma_t)}{\pi_t(a)}, \nabla F(q_t)\right) \\
&\leq \frac{\log(n/d)}{\eta} + \frac{1}{\eta} \sum_{t=1}^n \Breg_\star\left(\nabla F(q_t) - \frac{\eta \hat v(a_t, \sigma_t)}{\pi_t(a)}, \nabla F(q_t)\right) \,.
\end{align*}
Therefore, by the definition of $q_t$, $\hat v_t$, $\pi_t$ and $\Lambda_{q_t \eta}^\star$, the regret of \cref{alg:exp-by-opt} is bounded by
\begin{align}
\Reg_n
&\leq n\epsilon + n\epsilon_\cD + \frac{m \log(n/m)}{\eta} + \E\left[\sum_{t=1}^n \Lambda_{q_t \eta}^\star\right]  \,.
\label{eq:dual-regret}
\end{align}

\paragraph{Step 2: Minimax duality}
On the path to relating the information ratio and $\Lambda^\star_{q\eta}$ we need to apply minimax duality to the optimisation problem that
defines $\Lambda^\star_{q\eta}$.
Let $\cW = \sP(\cK \times \cD)$.
Given a distribution $\xi \in \cW$ and $\pi \in \ri(\Delta_k)$, let $\bbP_\xi^\pi$ be a probability measure on some measurable space
carrying random elements $X, P, A, Z, M, S \in \cK \times \cD \times [k] \times \cZ \times [m] \times \Sigma$ 
where $(\bbP^\pi_\xi)_{X, P, A} = \xi \otimes \pi$ and $(\bbP^\pi_\xi)_{M|X,P} = P$ a.s.\ and $(\bbP^\pi_\xi)_{Z|X,P} = X$ a.s.\ and $S = \cS(A, Z)$  
and $M \bot A$ and $Z \bot A$.
The expectation operator with respect to $\bbP^\pi_\xi$ is $\E^\pi_\xi$.
Let $q \in \cD$ and $\eta > 0$ and $\Delta_{k,\delta} = \Delta_k \cap [\delta, 1]^k$.
We claim that by repeating the proof of Theorem 9 in the paper by \cite{LG20}, 
\begin{align}
\Lambda_{q \eta}^\star \leq \inf_{\delta > 0} \sup_{\xi \in \cW} \min_{\pi \in \Delta_{k,\delta}} 
\E^\pi_\xi\left[\V(\pi, X) - \ip{P, \Vm(X)} - \frac{1}{\eta} \Breg\left(\E[P|A, S], \E[P]\right)\right]\,.
\label{eq:claim}
\end{align}
For completeness we summarise the argument here, which is slightly easier when adapted to our setting because $F$ and its gradients are bounded on $\cD$.
Remember that $\cE$ is the space of functions from $[k] \times \Sigma$ to $\R^m$.
Let $\cE_b \subset \cE$ be the set of all $\hat v \in \cE$ with $\norm{\hat v(a, s)}_\infty \leq 2n/\eta$ for all $a \in [k]$ and $s \in \Sigma$. 
The set $\cE_b$ is compact with respect
to the product topology by Tychonoff's theorem.
Abbreviate $\cH_q(x) = \Breg_\star(\nabla F(q) - x, \nabla F(q))$, which is convex.
Then by Sion's minimax theorem (see aforementioned paper for topoligical details), for any $\delta > 0$,
\begin{align}
\Lambda_{q \eta} 
&= \inf_{\substack{\pi \in \ri(\Delta_k) \\ \hat v \in \cE}} \sup_{\xi \in \cW}
\E^\pi_\xi \Bigg[\V(\pi, X) - \ip{P, \Vm(X)} + \bip{P - q, \frac{\hat v(A, S)}{\pi(A)}} 
  + \frac{1}{\eta} \cH_q\left(\frac{\eta \hat v(A, S)}{\pi(A)}\right)\Bigg] \nonumber \\
&\leq \min_{\substack{\pi \in \Delta_{k,\delta} \\ \hat v \in \cE_b}} \sup_{\xi \in \cW}
\E^\pi_\xi \Bigg[\V(\pi, X) - \ip{P, \Vm(X)} + \bip{P - q, \frac{\hat v(A, S)}{\pi(A)}} + \frac{1}{\eta} \cH_q\left(\frac{\eta \hat v(A, S)}{\pi(A)}\right)\Bigg] \nonumber \\
&= \sup_{\xi \in \cW} \min_{\substack{\pi \in \Delta_{k,\delta} \\ \hat v \in \cE_b}} 
\E^\pi_\xi \Bigg[\V(\pi, X) - \ip{P, \Vm(X)} + \bip{P - q, \frac{\hat v(A, S)}{\pi(A)}} + \frac{1}{\eta} \cH_q\left(\frac{\eta \hat v(A, S)}{\pi(A)}\right)\Bigg] \,.
\label{eq:sion}
\end{align}
Note that in the application of Sion's we are using the fact that $\pi \mapsto \V(\pi, x)$ is convex for all $x$, which is the only distinction between the argument
above and those of \cite{LG20}, where the corresponding term was linear.
Letting $\xi \in \cW$ and $\pi \in \Delta_{k,\delta}$ be fixed and
\begin{align*}
\hat v(a, s) = 
\begin{cases}
\frac{\pi(a)}{\eta} \left(\nabla F(q) - \nabla F(\E^\pi_\xi[P|S = s, A = a])\right) & \text{if } \bbP^\pi_\xi(S = s, A = a) > 0 \\
\zeros & \text{otherwise} \,,
\end{cases}
\end{align*}
which is chosen to minimise the bound on the right-hand side \cref{eq:sion}.
Note that $\hat v \in \cE_b$ follows from the fact that $q$ and $\E[P|S,A]$ are in $\cD$ and hence when $\bbP^\pi_\xi(S = s, A = a) > 0$,
\begin{align*}
\norm{\hat v(a, s)}_\infty \leq \frac{1}{\eta} \Big(\norm{\nabla F(q)}_\infty + \norm{\nabla F(\E[P|S=s,A=a])}_\infty \Big)
\leq \frac{2n}{\eta}\,,
\end{align*}
where we used also that $\nabla F(q) = -1/q$.
When $\bbP^\pi_\xi(S = s, A = a) = 0$, then $\norm{\hat v(a, s)}_\infty = 0$ by definition.
Substituting the definition of $\hat v$ into the second two terms of \cref{eq:sion} and repeating the calculation in Equation (6) of \cite{LG20},
\begin{align*}
\E^\pi_\xi\left[\bip{P - q, \frac{\hat v(A, S)}{\pi(A)}} + \frac{1}{\eta} \cH_q\left(\frac{\eta \hat v(A, S)}{\pi(A)}\right)\right] 
\leq -\frac{1}{\eta}\E^\pi_\xi\left[\Breg(\E^\pi_\xi[P|S,A], \E^\pi_\xi[P])\right] \,.
\end{align*}
Since $\delta > 0$ was arbitrary, \cref{eq:claim} follows by combining the above with \cref{eq:sion}.

\paragraph{Step 3: From Bregman divergence to information gain}
For $\xi \in \cW$, let $\mu_\xi \in \sP_m(\cK)$ be the distribution with the same law as $\E^\pi_\xi[X|M]$. 
We claim that the expected Bregman divergence in \cref{eq:claim} is $\I(\pi, \mu_\xi)$, which follows from Bayes' law:
\begin{align*}
&\E^\pi_\xi[\Breg(\E^\pi_\xi[P|A,S], \E^\pi_\xi[P])]
= \E^\pi_\xi[F(\E[P|A,S]) - F(\E^\pi_\xi[P])] \\
&\quad= \E\left[\sum_{\alpha=1}^m \log\left(\frac{\bbP^\pi_\xi(M = \alpha)}{\bbP^\pi_\xi(M = \alpha|A,S)}\right)\right] \\
&\quad= \sum_{\alpha=1}^m \sum_{a=1}^k \pi(a) \sum_{s \in \Sigma} \bbP^\pi_\xi(S = s|A=a) \log\left(\frac{\bbP^\pi_\xi(M = \alpha)}{\bbP^\pi_\xi(M = \alpha|S = s, A = a)}\right) \\
&\quad= \sum_{\alpha=1}^m \sum_{a=1}^k \pi(a) \sum_{s \in \Sigma} \bbP^\pi_\xi(S = s|A=a) \log\left(\frac{\bbP^\pi_\xi(M = \alpha) \bbP^\pi_\xi(S=s|A=a)}{\bbP^\pi_\xi(S = s|M = \alpha, A = a) \bbP^\pi_\xi(M = \alpha|A = a)}\right) \\
&\quad= \sum_{\alpha=1}^m \sum_{a=1}^k \pi(a) \sum_{s \in \Sigma} \bbP^\pi_\xi(S = s|A=a) \log\left(\frac{\bbP^\pi_\xi(S=s|A=a)}{\bbP^\pi_\xi(S = s|M = \alpha, A = a)}\right) \\
&\quad= \sum_{\alpha=1}^m \sum_{a=1}^k \pi(a) \KL(\bbP^\pi_\xi(S = \cdot|A=a), \bbP^\pi_\xi(S = \cdot|A=a,M = \alpha)) \\ 
&\quad= \sum_{\alpha=1}^m \sum_{a=1}^k \pi(a) \KL(\cS(a, \E^\pi_\xi[X]), \cS(a, \E^\pi_\xi[X|M = \alpha])) \\
&\quad= \I(\pi, \mu_\xi)\,,
\end{align*}
where in this calculation we have adopted the measure-theoretic convention that $0 \times \infty = 0$.
Note that because $P \in \cD$, $\bbP(M = \alpha) > 0$ for all $\alpha \in [m]$.

\paragraph{Step 4: Introducing the information ratio}
Combining the pieces from the previous steps,
\begin{align*}
\Lambda_{q \eta}^\star 
&\leq \sup_{\xi \in \cW} \min_{\pi \in \Delta_{k,\delta}} \E^\pi_\xi\left[\V(\pi, X) - \ip{P, \Vm(X)} - \frac{1}{\eta} \I(\pi, \mu_\xi)\right] \\
&= \sup_{\xi \in \cW} \min_{\pi \in \Delta_{k,\delta}} \E^\pi_\xi\left[\E^\pi_\xi[\V(\pi, X)|M] - \ip{P, \Vm(X)} - \frac{1}{\eta} \I(\pi, \mu_\xi)\right] \\
&\leq \sup_{\xi \in \cW} \min_{\pi \in \Delta_{k,\delta}} \E^\pi_\xi\left[\V(\pi, \E^\pi_\xi[X|M]) - \ip{P, \Vm(X)} - \frac{1}{\eta} \I(\pi, \mu_\xi)\right] \\
&\leq \sup_{\xi \in \cW} \min_{\pi \in \Delta_{k,\delta}} \E^\pi_\xi\left[\V(\pi, \E^\pi_\xi[X|M]) - \V_\star(\E^\pi_\xi[X|M]) - \frac{1}{\eta} \I(\pi, \mu_\xi)\right] \\
&= \sup_{\xi \in \cW} \min_{\pi \in \Delta_{k,\delta}} \left(\Delta(\pi, \mu_\xi) - \frac{1}{\eta} \I(\pi, \mu_\xi)\right) \,,
\end{align*}
where in the first inequality we used the concavity of $x \mapsto \V(\pi, x)$ and the second follows because
\begin{align*}
\E^\pi_\xi[\ip{P, \V(X)}]
&= \E^\pi_\xi[\E^\pi_\xi[\ip{P, \Vm(X)}|X,P]]
= \E^\pi_\xi[\E^\pi_\xi[\Vm(X)_M|X, P]] \\
&= \E^\pi_\xi[\E^\pi_\xi[\Vm(X)_M | M]] 
= \E^\pi_\xi[\Vm(\E^\pi_\xi[X|M])_M] 
\geq \E^\pi_\xi[\V_\star(\E^\pi_\xi[X|M])]\,.
\end{align*}
Since $\delta > 0$ is arbitrary,
\begin{align*}
\Lambda_{q\eta}^\star
&\leq \inf_{\delta>0} \sup_{\xi \in \cW} \min_{\pi \in \Delta_{k,\delta}} \left(\Delta(\pi, \mu_\xi) - \frac{1}{\eta} \I(\pi, \mu_\xi)\right) \\
&\leq \inf_{\delta>0} \sup_{\xi \in \cW} \min_{\pi \in \Delta_k} \left(\Delta((1 - k \delta) \pi + k \ones, \mu_\xi) - \frac{1}{\eta} \I((1 - k \delta)\pi + \delta \ones, \mu_\xi)\right) \\
&\leq \inf_{\delta>0}\left[k\delta + (1 - k\delta)\sup_{\xi \in \cW} \min_{\pi \in \Delta_k} \left(\Delta(\pi, \mu_\xi) - \frac{1}{\eta} \I(\pi, \mu_\xi) \right) \right] \,,
\end{align*}
where we used convexity of $\pi \mapsto \Delta(\pi, \mu_\xi)$ and that $\Delta(\pi, x) \in [0,1]$ and $\pi \mapsto \I(\pi, \mu_\xi)$ is linear and non-negative.
By the definition of the information ratio and Lemma~\ref{lem:exist}, for any $\xi$ there exists a policy $\pi \in \Delta_k$ such that
\begin{align*}
\Delta(\pi, \mu_\xi) \leq (\Psi_\star^\lambda)^{1/\lambda} \I(\pi, \mu_\xi)^{1/\lambda}
\end{align*}
Therefore,
\begin{align*}
\Lambda_{q\eta}^\star
&\leq \inf_{\delta > 0} \left[k \delta + (1 - k\delta) \sup_{\xi \in \cW} \left((\Psi_\star^\lambda)^{1/\lambda} \I(\pi, \mu_\xi)^{1/\lambda} - \frac{\I(\pi, \mu_\xi)}{\eta}\right)\right] 
\leq \frac{\lambda - 1}{\lambda} \left(\frac{\eta \Psi_\star^\lambda}{\lambda}\right)^{1/(\lambda - 1)} \,,
\end{align*}
where the second inequality follows by optimising the worst possible value of $\I(\pi, \mu_\xi)$ and because of the $\inf$ over $\delta > 0$.
Combining the above with \cref{eq:dual-regret} completes the proof.

\section{Comparing the notions of regret}\label{sec:compare}

\newcommand{\lambdaS}{\lambda_\star^{\scriptsize \textrm{s}}}
\newcommand{\lambdaR}{\lambda_\star^{\scriptsize \textrm{r}}}

There is a subtle connection between the Rustichini and the standard regret. By \cref{thm:main}, the minimax asymptotic regret for both settings is determined by
the minimax information ratio and the resulting value of $\lambda_\star$. 
Given a game $\cG = (\cL, \cS)$, let $\lambdaS(\cG)$ be the value of $\lambda_\star$ in game $\cG$ when the standard regret is used
and $\lambdaR(\cG)$ be its value when the Rustichini regret is used.
Throughout this section, assume that $\cZ$ is finite and let $\V$ and $\V_\star$ be the Rustichini versions as defined in \cref{eq:rustichini}.

By our main theorem and the upper bounds by \citep{KP17} we know that $\lambdaR(\cG) \leq 3$ for all finite games. The classification theorem for finite partial monitoring \citep{BFPRS14} 
and our Theorem~\ref{thm:main} show that 
$\lambdaS(\cG) \in \{1, 2, 3, \infty\}$. Games with $\lambdaS(\cG) = \infty$ are called
hopeless and the learner suffers linear minimax standard regret. So the Rustichini and standard regret are not the same on these games. An example is the game
in \cref{sec:exotic}, for which $\lambdaS(\cG) = \infty$ and $\lambdaR(\cG) = 7/3$.
A natural question is whether or not $\lambdaS(\cG) = \lambdaR(\cG)$ when $\lambdaS(\cG) \in \{1,2,3\}$.
The answer is obviously yes for all games where $x \rel y$ if and only if $x = y$.
The answer also turns out to be yes for the large class of non-degenerate games. For degenerate games there can be subtle differences.

To explain, we need to remind you about the classification of actions and games in the standard setting.
We assume some familiarity with the classification theorem for finite games \citep{BFPRS14} and the subtle definitions
of the cell decomposition and observability. Detailed intuition to complement the definitions below can be found in the literature \citep{LS20bandit-book}.
The cell decomposition for the standard regret is defined as follows. The cell associated with action $a$ is
\begin{align*}
\cC_a = \left\{x \in \cK : \cL(a, x) = \min_{b \in [k]} \cL(b, x)\right\}\,.
\end{align*}
The dimension of a cell is the dimension of its affine hull. Actions $a$ with $\dim(\cC_a) = \dim(\cK)$ are called Pareto optimal.
Actions $a$ with $\cC_a = \emptyset$ are called dominated and the remaining actions are called degenerate.
A pair of actions $a$ and $b$ are called globally observable if there exists a function $v : [k] \times \Sigma \to \R$ such that
\begin{align}
\cL(a, z) - \cL(b, z) = \sum_{c \in [k]} v(c, \cS(c, z)) = \cL(a, z) - \cL(b, z) \text{ for all } z \in \cZ\,,
\label{eq:est}
\end{align}
which just means that the loss differences between $a$ and $b$ can be estimated from the signals when playing all actions.
The game is called globally observable if all pairs of Pareto optimal actions are globally observable.
Note, there is no need to estimate the losses of dominated or degenerate actions because they cannot be uniquely optimal.
This is the subtlety that causes problems.

\begin{proposition}\label{prop:easy}
Suppose that all pairs of actions (not just Pareto optimal) are globally observable. Then $\lambdaS(\cG) = \lambdaR(\cG)$.
\end{proposition}

\begin{proof}
By the definition of global observability, there exists a function $\omega : \cK \times \cK \to \R$ such that for all actions $a$,
\begin{align*}
\cL(a, y) = \cL(a, x) + \omega(x, y) \,.
\end{align*}
Hence the definitions of the instantaneous regret (\cref{eq:regret}) for both settings coincide. Since the information gain is always the same, the minimax information ratios
also coincide and so $\lambdaS(\cG) = \lambdaR(\cG)$.
\end{proof}

With a small modification we can extend this result to games with dominated actions.

\begin{proposition}
For all globally observable games $\cG$ with no degenerate actions, 
$\lambdaS(\cG) = \lambdaR(\cG)$.
\end{proposition}

\begin{proof}
Let $\cG = (\cL, \cS)$ be a globally observable game with no degenerate actions, but possibly some dominated actions.
Assume without loss of generality that actions $1,\ldots,p$ are Pareto optimal and actions $p+1,\ldots,k$ are dominated.
Suppose that $a$ is a dominated action, which means that $\cL(a, x) < \min_{b \in [k]} \cL(b, x)$ for all $x \in \cK$.
By von Neummann's minimax theorem and compactness of $\cK$ there exists an $\epsilon > 0$ such that for all dominated actions $a$,
\begin{align*}
\max_{\pi \in \Delta_k : \spt(\pi) \subset [p]} \min_{x \in \cK} \cL(a, x) - \cL(\pi, x) 
&= \min_{x \in \cK} \max_{\pi \in \Delta_k : \spt(\pi) \subset [p]} \cL(a, x) - \cL(\pi, x) 
\geq \epsilon\,.
\end{align*}
Hence for each dominated action $a \in [k]$ there exists a distribution $\rho_a$ over Pareto optimal actions such that for all $x \in \cK$,
\begin{align*}
\cL(\rho_a, x) + \epsilon \leq \cL(a, x)\,.  
\end{align*}
Next, define a new loss function by 
\begin{align*}
\tilde \cL(a, x) = 
\begin{cases}
\cL(a, x) & \text{if } a \leq p \\
\cL(\rho_a, x) + \epsilon \leq \cL(a, x) & \text{if } a > p \,.
\end{cases}
\end{align*}
The game defined by $\tilde \cG = (\tilde \cL, \cS)$ has the same dominated actions as the original game and the same
loss function for the non-dominated actions. Since the signal function is the same, the modified game has the same classification as the original game.
On the other hand, 
\begin{align*}
\max_{y : y \rel x} \tilde \cL(\pi, y) - \min_{\rho \in \Delta_k} \max_{y : y \rel x} \tilde \cL(\rho, y)
&= \max_{y : y \rel x} \tilde \cL(\pi, y) - \min_{\rho \in \Delta_k} \max_{y : y \rel x} \cL(\rho, y) \\
&\leq \max_{y : y \rel x} \cL(\pi, y) - \min_{\rho \in \Delta_k} \max_{y : y \rel x} \cL(\rho, y)\,.
\end{align*}
Hence the instantaneous regret function using the Rustichini definition is smaller in the modified game than in the original. And since the signal structure
is the same, the minimax information ratio is also smaller in the modified game than the original.
Furthermore, by the definition of the modified game all pairs of actions are now globally observable, so by Proposition~\ref{prop:easy}, 
\begin{align*}
\lambdaR(\cG) \geq \lambdaR(\tilde \cG) = \lambdaS(\tilde \cG) = \lambdaS(\cG)
\end{align*}
For the upper bound we need to consider the cases. We already know that $\lambdaR(\cG) \leq 3$, so we are done when $\lambdaS(\cG) = 3$.
Games where $\lambdaS(\cG) = 1$ have a single action that is optimal everywhere, so in this case $\lambdaR(\cG) = 1$ is trivial as well.
We need to upper bound $\lambdaR(\cG)$ for games where $\lambdaS(\cG) = 2$, which are the so-called non-trivial locally observable games.
For these games it is known that there exists a policy with optimal regret (up to constant factors) that only plays actions $a$ that are Pareto optimal
or where $\cL(a, \cdot) = (1 - \alpha) \cL(b, \cdot) + \alpha \cL(c, \cdot)$ for some $\alpha \in (0,1)$ and with $b$ and $c$ Pareto optimal \citep{LS19pmsimple}.
Let $\tilde \cG$ be the game obtained from $\cG$ by eliminating the actions that do not satisfy the above criteria. 
Since all the eliminated actions are degenerate, this can only increase the Rustichini instantaneous regret and hence $\lambdaR(\cG) \leq \lambdaR(\tilde \cG)$.
But in $\tilde \cG$ all pairs of actions are globally observable and hence
\begin{align*}
\lambdaR(\cG) \leq \lambdaR(\tilde \cG) = \lambdaS(\tilde \cG) = \lambdaS(\cG)\,,
\end{align*}
where the final equality follows because removing the actions not satisfying the above criteria cannot change the classification of the game.
\end{proof}

When there are degenerate actions it is possible that $\lambdaR(\cG) \neq \lambdaS(\cG)$.
Consider the following example:
\begin{align*}
\cL &= \left[
\begin{matrix}
1 & 0 & 2 & 2 \\
0 & 1 & 2 & 2 \\
2 & 2 & 2 & 2 \\
2 & 2 & 2 & 2 
\end{matrix}
\right] & 
\cS &= \left[
\begin{matrix}
0 & 0 & 0 & 0 \\
0 & 0 & 0 & 0 \\
1 & 0 & 1 & 0 \\
0 & 1 & 1 & 0
\end{matrix}
\right] \,.
\end{align*}
The adversary has four choices, so $\cK$ is the probability simplex with four corners.
Actions 3 and 4 are degenerate. They are only optimal if $x_1 = x_2 = 0$. They are the only actions that reveal information.
Action 1 is optimal if $x_2 \geq x_1$ and action 2 is optimal otherwise. Equivalently, action 1 is optimal if $x_2 + x_3 \geq x_1 + x_3$. 
Both sides of this inequality can be estimated by playing actions $3$ and $4$. You can check that this game is globally observable but not locally observable, so
$\lambdaS(\cG) = 3$.
Let $p(x) = x_1 + x_3$ and $q(x) = x_2 + x_3$.
The adversary's optimal decision $y$ with $y \rel x$ is to play $y_3 = \min(p(x), q(x))$ and $y_1 = p(x) - y_3$ and $y_2 = q(x) - y_3$ with the remaining mass on $y_4$.
In other words, the adversary chooses the $y$ with $y \rel x$ with as much mass on $x_3$ and $x_4$ as possible.
Calculating, we have
\begin{align*}
\V_\star(x) = 2 - 2(p(x) + q(x)) + 4 \min(p(x), q(x))\,.
\end{align*}
On the other hand, if $\pi$ is a Dirac on any convex combination of actions $3$ and $4$, then obviously $\V(\pi, x) = 2$ and $\Delta(\pi, x) = 2 |p(x) - q(x)|$.
By the proof of our lower bound in \cref{thm:main}, it suffices to prove an upper bound on $\lambdaR(\cG)$ by constructing a policy with small regret against a stochastic adversary.
In this case the simple explore-then-commit policy that plays actions $3$ and $4$ uniformly until it identifies statistically the sign of $p(x) - q(x)$ and
then plays optimally has regret of order $n^{1/2}$ up to logarithmic factors.
Hence, $\lambdaR(\cG) \leq 2$. Since $\lambdaS(\cG) = 3$, this gives an example where the two quantities are different.
At an intuitive level, the difference arises because in the Rustichini definition of the regret the adversary is compelled to play the worst case over the indistinguishable outcomes.
When degenerate actions are involved, this can be the difference between a degenerate action being near-optimal or not.
\end{document}